\documentclass[12pt]{amsart}
\usepackage[utf8]{inputenc}

\usepackage[margin=1.5in]{geometry}

\usepackage{amsmath}
\usepackage{amsfonts}
\usepackage{amssymb}
\usepackage{amsthm}
\usepackage{mathtools}
\usepackage{caption}
\usepackage{subcaption}
\usepackage{bbm}
\usepackage[export]{adjustbox}
\usepackage{overpic}

\usepackage[all]{xy}

\usepackage{tikz-cd}
\usetikzlibrary{matrix}
\usepackage{graphicx} 
\usepackage{epstopdf}

\usepackage[linktocpage]{hyperref}
\hypersetup{
    colorlinks=true,
    linkcolor=blue,
    citecolor=blue,      
    urlcolor=blue,
}

\newtheorem{theorem}{Theorem}
\newtheorem{definition}[theorem]{Definition}
\newtheorem{proposition}[theorem]{Proposition}
\newtheorem{lemma}[theorem]{Lemma}

\newtheorem{corollary}[theorem]{Corollary}

\newtheorem{remark}[theorem]{Remark}
\newtheorem{conjecture}[theorem]{Conjecture}
\newtheorem{question}[theorem]{Question}

\numberwithin{equation}{section}
\numberwithin{theorem}{section}

\usepackage[english]{babel}

\usepackage{hyphenat}

\usepackage[backend=bibtex,style=alphabetic,maxalphanames=4,maxnames=4]{biblatex}

\renewbibmacro{in:}{}

\DeclareDelimFormat[bib,biblist]{nametitledelim}{\addcomma\space}

\DeclareFieldFormat*{title}{\mkbibitalic{#1}\addcomma}
\DeclareFieldFormat*{journaltitle}{#1}
\DeclareFieldFormat*{volume}{\mkbibbold{#1}}
\DeclareFieldFormat{pages}{#1}
\DeclareFieldFormat[misc]{date}{preprint {#1}}
\DeclareFieldFormat{mr}{%
  MR\addcolon\space
  \ifhyperref
    {\href{http://www.ams.org/mathscinet-getitem?mr=MR#1}{\nolinkurl{#1}}}
    {\nolinkurl{#1}}}
    
\AtEveryBibitem{
  \clearfield{url}
  \clearfield{number}
  \clearfield{doi}
  \clearfield{issn}
  \clearfield{isbn}
  \clearfield{eprintclass}
}
\AtEveryBibitem{\ifentrytype{book}{\clearfield{pages}}{}}

\bibliography{hecke}

\usepackage{fancyhdr}
\title{Higher-dimensional Heegaard Floer homology and Hecke algebras}
\author{Ko Honda}
\address{University of California, Los Angeles, Los Angeles, CA 90095}
\email{honda@math.ucla.edu} \urladdr{http://www.math.ucla.edu/\char126 honda}

\author{Yin Tian}
\address{School of Mathematical Sciences, Beijing Normal University; 
Laboratory of Mathematics and Complex Systems, Ministry of Education, Beijing 100875, China}
\email{yintian@bnu.edu.cn} \urladdr{}

\author{Tianyu Yuan}
\address{Beijing International Center for Mathematical Research, Peking University, Beijing 100871, China}
\email{ytymath@pku.edu.cn} \urladdr{}

\date{\today}

\keywords{Higher-dimensional Heegaard Floer homology, Hecke algebra, skein relation, categorification}

\subjclass[2010]{Primary 53D10; Secondary 53D40.}

\thanks{KH supported by NSF Grants DMS-1406564, DMS-154914, and DMS-2003483. YT supported by NSFC 11971256.}

\begin{document}

\maketitle

\begin{abstract}
Given a closed oriented surface $\Sigma$ of genus greater than 0, we construct a map $\mathcal{F}$ from the wrapped higher-dimensional Heegaard Floer homology of the cotangent fibers of $T^*\Sigma$ to the Hecke algebra associated to $\Sigma$ and show that $\mathcal{F}$ is an isomorphism of algebras. 
We also establish analogous results for punctured surfaces.
\end{abstract}

\tableofcontents

\section{Introduction}

In \cite{colin2020applications}, Colin, Honda and Tian developed the foundations of the higher-dimensional Heegaard Floer homology (HDHF).  
HDHF is supposed to model the Fukaya category of the Hilbert scheme (or flag Hilbert scheme) of points on a Liouville domain and was used to analyze symplectic fillability questions in higher-dimensional contact topology.
As an application, the HDHF of the $4$-dimensional Milnor fibration of type $A$ provided an invariant of links in $S^3$. 
This invariant is a close cousin of symplectic Khovanov homology \cite{seidel2006link, manolescu2006nilpotent} and especially its cylindrical reformulation \cite{mak2020fukayaseidel}. 
As a categorified quantum invariant, Khovanov homology is directly related to the categorification of quantum groups and to various Hecke algebras, including affine Hecke algebras and quiver Hecke algebras (also called KLR algebras) \cite{chuang2008derived, rouquier20082kacmoody, khovanov2009diagrammatic}. 
There are several approaches to these Hecke algebras from the point of view of geometric representation theory: The affine Hecke algebra can be realized as the equivariant $K$-theory of the ordinary Steinberg variety \cite{kazhdan1987proof, chriss2009representation, lusztig1998bases}; the double affine Hecke algebra admits a similar realization in terms of the loop Steinberg variety \cite{vasserot2005induced}. 
Moreover, the rational double affine Hecke algebra of type $A$ is closely related to the Hilbert scheme of points on $\mathbb{C}^2$; see \cite{gordon2005rational,kashiwara2008}. 
It is therefore natural to ask: 
~\\

\noindent
\textit{Question:} Is there a symplectic geometry interpretation of the various Hecke algebras?
~\\

The goal of this paper is to use the HDHF of cotangent bundles of oriented surfaces
to give an answer for the various Hecke algebras of type $A$, including the finite, affine, and double affine Hecke algebras (abbreviated DAHA and also called Cherednik algebras).  
We consider two cases:
\begin{enumerate} 
\item a closed oriented surface $\Sigma$ of genus $g>0$; 
\item a surface $\mathring{\Sigma}$ which is obtained from a closed oriented surface of genus $g\geq 0$ by removing a finite number ($>0$) of punctures.
\end{enumerate}
More precisely, we realize the various Hecke algebras as the HDHF of the disjoint cotangent fibers of $T^*\Sigma$ (or $T^*\mathring{\Sigma}$). 

\begin{center}
    \begin{tabular}{|c|c|} 
        \hline
        {\em Type of Hecke algebra} & {\em Surface} \\ 
        \hline
         finite Hecke algebra & open disk \\ 
        \hline
        affine Hecke algebra & cylinder \\
        \hline
        DAHA & torus\\ 
        \hline
    \end{tabular}
\end{center}

\begin{remark}
    Note that we exclude the case of a sphere which is more complicated since the homology of its based loop space is not supported in degree zero. We hope to revisit this in a future paper.
\end{remark}

\begin{remark}
    Recently, Ben-Zvi, Chen, Helm and Nadler identified the affine Hecke algebra with the endomorphism algebra of the coherent Springer sheaf, for any reductive algebraic group \cite[Theorem 1.7]{ben2020coherent}. 
    It would be interesting to study the connection between the algebro-geometric realization of the affine Hecke algebra of type $A$ and our symplectic geometry one.  
\end{remark}

\begin{remark}
    In this paper we consider $T^*\Sigma$ as a symplectic manifold, not as a holomorphic symplectic manifold corresponding to the Riemann surface $\Sigma$, which is more natural in many contexts (e.g., \cite[Chapter 7]{nakajima1999lectures}).
\end{remark}

In addition to the definition of HDHF, this work crucially depends on three key ingredients: 
\begin{enumerate} 
    \item the relationship between the wrapped Floer cochain complex of a cotangent fiber and chains on the based loop space of the base due to Abbondandolo and Schwarz \cite{abbondandolo2010floer} and Abouzaid \cite{abouzaid2012wrapped};
    \item an interpretation of the HOMFLY skein relation in terms of holomorphic curve counting due to Ekholm and Shende \cite{ekholm2021skeins};  
    \item a topological description of DAHA of $\mathfrak{gl}_{\kappa}$ as a {\em braid skein algebra} due to Morton and Samuelson \cite{morton2021dahas} (we have been informed that this was also known to Cherednik but unpublished). 
\end{enumerate}

Our first ingredient is the result of Abbondandolo-Schwarz \cite{abbondandolo2010floer} which states that the wrapped Floer cochain complex $CW^*(T^*_q\Sigma)$ of a cotangent fiber $T^*_q\Sigma$ and the chain complex $C_{-*}(\Omega_q\Sigma)$ of the based loop space of $\Sigma$ are isomorphic as graded algebras on the cohomology level. 
Abouzaid \cite{abouzaid2012wrapped} further improved this to an $A_\infty$-equivalence on the chain level.

In this paper we investigate its generalization to HDHF. 
More precisely, we consider $CW(\sqcup_{i=1}^{\kappa}T_{q_i}^*\Sigma)$, the wrapped HDHF cochain complex of $\kappa$ disjoint cotangent fibers of $T^*\Sigma$; it can be given the structure of an $A_{\infty}$-algebra. 
The HDHF complex $CW(\sqcup_{i}T_{q_i}^*\Sigma)$ is defined over $\mathbb{Z}[[\hbar]]$, the ring of formal power series in $\hbar$, where $\hbar$ keeps track of the Euler characteristic of the holomorphic curves that are counted in the definition of the $A_{\infty}$-operations.
Since $\Sigma$ is a surface $\not=S^2$, $CW(\sqcup_{i}T_{q_i}^*\Sigma)$ is supported in degree zero and hence is an ordinary algebra. 

Our generalization of the based loop space of $\Sigma$ will be the based loop space of the unordered configuration space $\mathrm{UConf}_{\kappa}(\Sigma)$ of $\kappa$ points on $\Sigma$.  
Generalizing Abouzaid's map $CW^*(T^*_q\Sigma) \to C_{-*}(\Omega_q\Sigma)$, we define an evaluation map 
$$\mathcal{E}\colon CW(\sqcup_{i}T_{q_i}^*\Sigma) \to C_0(\Omega(\mathrm{UConf}_{\kappa}(\Sigma))) \otimes \mathbb{Z}[[\hbar]].$$
Here $C_0(\Omega(\mathrm{UConf}_{\kappa}(\Sigma)))$ is the $0$th chain space of the based loop space of $\mathrm{UConf}_{\kappa}(\Sigma)$ and all tensor products are over $\mathbb{Z}$, unless indicated otherwise.  
The map $\mathcal{E}$ is given by counting curves of ``Heegaard Floer type''; the precise definition will be given in Section \ref{subsection-ev}.
This map however fails to be a homomorphism of algebras due to an additional degeneration of curves: the nodal degeneration. 
This phenomenon was recently clarified by Ekholm and Shende \cite{ekholm2021skeins} (building on unpublished work of Fukaya on the relationship of higher-genus Lagrangian Floer homology and string topology): It is the HOMFLY skein relation that controls the boundaries of $1$-dimensional moduli spaces of varying Euler characteristics; see Figure \ref{fig-skein-brane}. 

Starting with the map $\mathcal{E}$, taking the homology of both sides and quotienting out by the HOMFLY skein relation, we obtain a map 
    \begin{equation}
        \mathcal{F}\colon HW(\sqcup_{i}T_{q_i}^*\Sigma)\to H_0(\Omega(\mathrm{UConf}_{\kappa}(\Sigma))) \otimes \mathbb{Z}[[\hbar]] / \mbox{\{the skein relation\}}.
        \label{eq-main1}
    \end{equation}
Here, $HW(\sqcup_{i}T_{q_i}^*\Sigma)$ denotes the homology of the wrapped HDHF complex $CW(\sqcup_{i}T_{q_i}^*\Sigma)$. 
The map $\mathcal{F}$ is an algebra homomorphism; see Proposition \ref{prop-algebra}.

At this point we observe that $H_0(\Omega(\mathrm{UConf}_{\kappa}(\Sigma)))$ is isomorphic to the group algebra of the surface braid group of $\Sigma$ over $\mathbb{Z}$. In \cite{morton2021dahas}, Morton and Samuelson defined the {\em braid skein algebra} $\mathrm{BSk}_\kappa(\Sigma)$, which is a quotient of the group algebra of the surface braid group over $\mathbb{Z}[s^{\pm1},c^{\pm1}]$ by the skein relation and the marked point relation; see Definition \ref{def-skein}. 
Here $s$ and $c$ are parameters that appear in the skein and marked point relations, respectively.   
We define the {\em surface Hecke algebra} $\mathrm{H}_\kappa(\Sigma)$ by reformulating the marked point relation as a $c$-deformed homotopy relation, and making a change of variables $\hbar=s-s^{-1}$; see Definition \ref{def-hecke}. 
We show that $\mathrm{H}_\kappa(\Sigma)$ and $\mathrm{BSk}_\kappa(\Sigma)$ are isomorphic, up to a change of variables. 

Motivated by the $c$-deformed homotopy relation, we consider $CW(\sqcup_{i}T_{q_i}^*\Sigma)_c$, the wrapped HDHF with a parameter $c$.  
Adding the parameter $c$ to the map in (\ref{eq-main1}), we obtain
    \begin{equation}
        \mathcal{F}\colon HW(\sqcup_{i}T_{q_i}^*\Sigma)_c \to \mathrm{H}_\kappa(\Sigma) \otimes_{\mathbb{Z}[\hbar]} \mathbb{Z}[[\hbar]],
        \label{eq-main}
    \end{equation}
which is still denoted $\mathcal{F}$ by abuse of notation.




We then apply the Abbondandolo-Schwarz result \cite{abbondandolo2006floer} to show that the restriction of $\mathcal{F}$ to $\hbar=0$ is an isomorphism.
The following is the main result of this paper and directly follows from the isomorphism of $\mathcal{F}|_{\hbar=0}$:


\begin{theorem}
    \label{thm-main}
    The map $\mathcal{F}$ in (\ref{eq-main}) is an isomorphism of algebras.
\end{theorem}


The main result from Morton-Samuelson \cite{morton2021dahas} is that the double affine Hecke algebra $\ddot{\mathrm{H}}_{\kappa}$ of $\mathfrak{gl}_{\kappa}$ is naturally isomorphic to $\mathrm{BSk}_\kappa(T^2)$. Hence we have:

\begin{corollary}
The algebra $HW(\sqcup_{i}T_{q_i}^*T^2)_c$ is isomorphic to the tensor product $\ddot{\mathrm{H}}_{\kappa}|_{\hbar=s-s^{-1}} \otimes_{\mathbb{Z}[\hbar]} \mathbb{Z}[[\hbar]]$.  
\end{corollary}

When the base $\mathring{\Sigma}$ has punctures, one can similarly define the HDHF $A_{\infty}$ algebra $CW(\sqcup_{i}T_{q_i}^*\mathring{\Sigma})$. 
It is possible to formally include a $c$-parameter, but we expect that it does not yield any extra information. 
The isomorphism (\ref{eq-main}) still holds and the corresponding surface Hecke algebra $\mathrm{H}_\kappa(\mathring{\Sigma})$ is isomorphic to the finite Hecke algebra $\mathrm{H}_{\kappa}$ and the affine Hecke algebra $\dot{\mathrm{H}}_{\kappa}$ of $\mathfrak{gl}_{\kappa}$ when $\mathring{\Sigma}$ is an open disk and a cylinder, respectively. 

\begin{corollary}
The algebra $HW(\sqcup_{i}T_{q_i}^*\mathring{\Sigma})$ is isomorphic to the tensor product $\mathrm{H}_{\kappa}|_{\hbar=s-s^{-1}} \otimes_{\mathbb{Z}[\hbar]} \mathbb{Z}[[\hbar]]$ (resp.\ $\dot{\mathrm{H}}_{\kappa}|_{\hbar=s-s^{-1}} \otimes_{\mathbb{Z}[\hbar]} \mathbb{Z}[[\hbar]]$), when $\mathring{\Sigma}$ is an open disk (resp.\ a cylinder). 
\end{corollary}


Returning to the discussion of categorification, the affine Hecke algebra is related to categorified quantum groups of type $A$ \cite{ariki1996decomposition}. 
This can be explained symplectically by noting that $T^*\mathring{\Sigma}$ for a cylinder $\mathring{\Sigma}$ is symplectomorphic to $\mathbb{R}^2 \times T^*S^1$. The latter naturally appears in the $4$-dimensional Milnor fibration and the HDHF approach to symplectic Khovanov homology \cite{CHT}.  

The isomorphism (\ref{eq-main}) holds only after tensoring with $\mathbb{Z}[[\hbar]]$. 
Nevertheless, we believe that the coefficient ring could be taken to be $\mathbb{Z}[\hbar]$.

\begin{conjecture}
    \label{conj-main}
    The algbera $HW(\sqcup_{i}T_{q_i}^*\Sigma)_c$ is well-defined over $\mathbb{Z}[\hbar]$ and Theorem \ref{thm-main} still holds over $\mathbb{Z}[\hbar]$. 
\end{conjecture}

There is some evidence for this conjecture:  In particular, direct computations of the second and third authors \cite{TY} (in preparation) show the well-definedness over $\mathbb{Z}[\hbar]$ when $\Sigma=\mathbb{R}^2$ and the surface Hecke algebra is isomorphic to the finite Hecke algebra.

\vskip.2in
\begin{question}
\label{question-change of variables}
    What is the geometric meaning of the change of variables $\hbar=s-s^{-1}$?
\end{question}
Question \ref{question-change of variables} is important from the perspective of representation theory. 
For instance, the affine Hecke algebra with the parameter $s$ has interesting modules, but the situation is not clear for the affine Hecke algebra with the parameter $\hbar$. 
A possible explanation of $\hbar=s-s^{-1}$ may require additional data of flat bundles which is well-studied in mirror symmetry and Fukaya categories.   

On the HDHF side, other $\kappa$-tuples of Lagrangians in $T^*\Sigma$ give rise to modules over $HW(\sqcup_{i}T_{q_i}^*\Sigma)_c$. It is interesting to look at the category of such modules and try to relate it to the Fukaya category of the Hilbert scheme $\mathrm{Hilb}^{\kappa}(T^*\Sigma)$.

\begin{remark}
    For any pair $(X, G)$, where $X$ is a smooth complex algebraic variety with an action of a finite group $G$, Etingof introduced a global analogue of the rational Cherednik algebra, and a Hecke algebra \cite{EtingofPavel}.  
    The two algebras are related by the KZ functor. 
    
    In the case of $(X,G)=(C^{\times n}, S_n)$, where $C$ is an algebraic curve, and the symmetric group $S_n$ acts on the product by permutation, the associated Hecke algebra is likely to be isomorphic to the braid skein algebra. This case was further studied by Finkelberg and Ginzburg via quantum Hamiltonian reduction \cite{Finkelberg}. 
    
    It is interesting to compare our symplectic geometry approach with their algebraic geometry one, and to look for a symplectic explanation of the category $\mathcal{O}$ of rational Cherednik algebras.  
\end{remark}

~\\
\noindent
\textit{Organization:} 
In Section \ref{section-hdhf}, we give a brief review of HDHF. We then restrict to the special case of cotangent bundles of surfaces and define $CW(\sqcup_{i}T_{q_i}^*\Sigma)$. 
In Section \ref{section-loop}, we review the chain complex of the based loop space and the results of \cite{abbondandolo2006floer,abouzaid2012wrapped}. 
In Section \ref{section-Hecke}, we discuss the based loop space of $\mathrm{UConf}_{\kappa}(\Sigma)$, the braid skein algebra $\mathrm{BSk}_\kappa(\Sigma)$, and its reformulation, the surface Hecke algebra $\mathrm{H}_\kappa(\Sigma)$. 
In Section \ref{section-c}, we define $CW(\sqcup_{i}T_{q_i}^*\Sigma)_c$, the HDHF with a parameter $c$. 
In Section \ref{section-f}, we construct the map $\mathcal{F}$ and prove that it is an isomorphism of algebras. 
In Section \ref{section-boundary}, we show that our results can be extended to surfaces with punctures.
~\\

\noindent \textit{Acknowledgements}. KH and YT thank Vincent Colin for \cite{colin2020applications} and also the follow-up work \cite{CHT} in which we give a different Hecke algebra corresponding to a surface and arising from $\Sigma\times T^*S^1$. YT thanks Peng Shan for numerous conversations and for bringing the work of Morton-Samuelson to our attention.

\section{Wrapped higher-dimensional Heegaard Floer cohomology} \label{section-hdhf}
We briefly review HDHF in Section \ref{subsection-reviewwrap} and then turn to the special case of disjoint cotangent fibers of $T^*\Sigma$ in Section \ref{subsection-wrap}.

\subsection{Review of HDHF}
\label{subsection-reviewwrap}
We refer the reader to \cite{colin2020applications} for more details. 

 Let $(X,\alpha)$ be a $2n$-dimensional completed Liouville domain, i.e., there exists a Liouville domain $X^c\subset X$ such that $X$ is obtained from $X^c$ by attaching the symplectization end $\left([0,\infty)_s\times \partial X^c, e^s\alpha|_{\partial X^c}\right)$.
    Let $\omega=d\alpha$ be the exact symplectic form on $X$.

\begin{definition}

    The objects of the $A_\infty$-category $\mathcal{F}_\kappa(X)$ are $\kappa$-tuples of disjoint exact Lagrangians which are equipped with relative spin structures and grading data, and, if not compact, are cylindrical at infinity. 

    Given two objects $L_i=L_{i1}\sqcup\dots\sqcup L_{i\kappa}$, $i=0,1$, whose components are mutually transverse, the morphism $\mathrm{Hom}_{\mathcal{F}_\kappa(X)}(L_0,L_1)=CF(L_0,L_1)$ is the free abelian group generated by all $\mathbf{y}=\{y_{1},\dots,y_{\kappa}\}$ where $y_{j}\in L_{0j}\cap L_{1\sigma(j)}$ and $\sigma$ is some permutation of $\{1,\dots,\kappa\}$. 
    The coefficient ring is set to be $\mathbb{Z}[[\hbar]]$. 
    The $A_\infty$-operations $\mu^m$, $m=1,2,\dots,$ will be defined by (\ref{eq-m_2}).
\end{definition}

To define the $A_\infty$ operation $\mu^m$, for $i=1,\dots,m$, let $L_i=\sqcup_{j=1}^\kappa L_{ij}$ so that $L_{i-1}$ and $L_{i}$ are transverse. Let $\mathbf{y}_i=\{y_{i1},\dots,y_{i\kappa}\}$ be a $\kappa$-tuple of points so that $y_{ij}\in L_{(i-1)j}\cap L_{i\sigma_i(j)}$ where $\sigma_i$ is some permutation of $\{1,\dots,\kappa\}$.

As in the cylindrical reformulation of Heegaard Floer homology by Lipschitz \cite{lipshitz2006cylindrical}, we also need an extra ``cylindrical'' direction to keep track of points in ``$\mathrm{Sym}^\kappa(X)$''. Specifically, as shown in Figure \ref{fig-base}, let $D$ be the unit disk in $\mathbb{C}$ and $D_m=D-\{p_0,\dots,p_m\}$, where $p_i\in\partial D$ are boundary marked points arranged counterclockwise. Let $\partial_i D_m$ be the boundary arc from $p_i$ to $p_{i+1}$. 
Let $\mathcal{A}_m$ be the moduli space of $D_m$ modulo automorphisms; we choose representatives $D_m$ of equivalence classes of $\mathcal{A}_m$ in a smooth manner (e.g., by setting $p_0=-i$ and $p_1=i$) and abuse notation by writing $D_m\in \mathcal{A}_m$.
We call $D_m$ the ``$A_\infty$ base direction''.

\begin{figure}[ht]
    \centering
    \includegraphics[width=5cm]{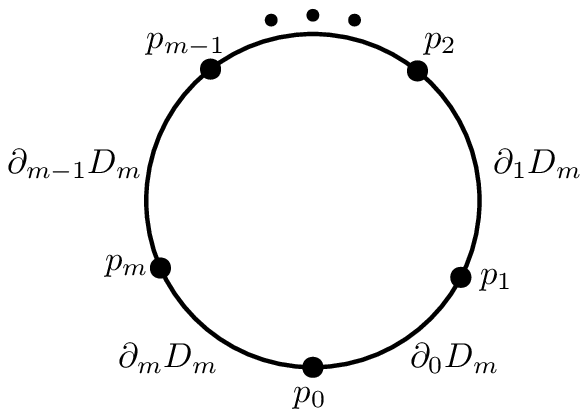}
    \caption{}
    \label{fig-base}
\end{figure}

The full ambient symplectic manifold is then $(D_m\times X,\,\Omega_m=\omega_m+\omega)$, where $\omega_m$ is an area form on $D_m$ which restricts to $ds_i\wedge dt_i$ on the strip-like end $e_i\simeq [0,\infty)_{s_i}\times[0,1]_{t_i}$ around $p_i$ for $i=1,\dots,m$ and $e_i\simeq (-\infty,0]_{s_i}\times[0,1]_{t_i}$ around $p_i$ for $i=0$. (As we approach the puncture $p_i$, $s_i\to +\infty$ for $i=1,\dots,m$ and $s_i\to -\infty$ for $i=0$.) Then we extend the Lagrangians to the $A_\infty$ base direction: for $i=0,\dots,m$, let $\tilde{L}_i=\partial_iD_m\times L_i$ and $\tilde{L}_{ij}=\partial_iD_m\times L_{ij}$. 
Let $\pi_X: D_m\times X\to X$ be the projection to $X$ and 
$\pi_{D_m}$ be the symplectic fibration
\begin{equation*}
    \pi_{D_m}\colon (D_m\times X,\Omega_m)\to (D_m,\omega_m).
\end{equation*}

Let $\mathcal{J}_{X,\alpha}$ be the set of $d\alpha$-compatible almost complex structures $J_X$ on $(X,\omega)$ that are asymptotic to an almost complex structure on $[0,\infty)_s\times \partial X^c$ that takes $\partial_s$ to the Reeb vector field of $\alpha|_{\partial X^c}$, takes $\ker \alpha|_{\partial X^c}$ to itself, and is compatible with $d\alpha|_{\partial X^c}$. 

There is a smooth assignment 
$D_m\mapsto J_{D_m}$, where $D_m\in \mathcal{A}_m$, such that:
\begin{enumerate}
    \item[(J1)] on each fiber $\pi_{D_m}^{-1}(p)=\{p\}\times X$, $J_{D_m}$ restricts to an element of $\mathcal{J}_{X,\alpha}$;
    \item[(J2)] $J_{D_m}$ projects holomorphically onto $D_m$;
    \item[(J3)] over each strip-like end $[0,\infty)_{s_i}\times[0,1]_{t_i}$, for $s_i$ sufficiently positive, or $(-\infty,0]_{s_0}\times[0,1]_{t_0}$, for $s_0$ sufficiently negative, $J_{D_m}$ is invariant in the $s_i$-direction and takes $\partial_{s_i}$ to $\partial_{t_i}$; when $m=1$, $J_{D_1}$ is invariant under $\mathbb{R}$-translation of the base and takes $\partial_{s_i}$ to $\partial_{t_i}$.
\end{enumerate}
One can inductively construct such an assignment for all $m\geq1$ in a manner which is (A) consistent with the boundary strata 
and (B) for which all the moduli spaces $\mathcal{M}({\bf y}_1,\dots, {\bf y}_m, {\bf y}_0)$, defined below, are transversely cut out. 
A collection $\{ J_{D_m}~|~D_m\in \mathcal{A}_m,~m\in \mathbb{Z}_{>0}\}$ satisfying (A) will be called a {\em consistent collection} of almost complex structures; if it satisfies (B) in addition, it is a {\em sufficiently generic} consistent collection.

\begin{remark}  
    To avoid cumbersome terminology, in what follows, when we say {\em sufficiently generic}, we mean that all the moduli spaces under consideration are transversely cut out.
\end{remark}

Let $\mathcal{M}(\mathbf{y}_1,\dots,\mathbf{y}_m,\mathbf{y}_0)$ be the moduli space of maps
\begin{equation*}
    u\colon (\dot F,j)\to(D_m\times X,J_{D_m}),
\end{equation*}
where $(F,j)$ is a compact Riemann surface with boundary, $\mathbf{p}_0,\dots,\mathbf{p}_m$ are disjoint $\kappa$-tuples of boundary marked points of $F$, $\dot F=F\setminus\cup_i \mathbf{p}_i$, and $D_m\in \mathcal{A}_m$, so that $u$ satisfies
\begin{align}
    \label{floer-condition}
    \left\{
        \begin{array}{ll}
            \text{$du\circ j=J_{D_m}\circ du$;}\\
            \text{each component of $\partial \dot F$ is mapped to a unique $\tilde{L}_{ij}$;}\\
            \text{$\pi_X\circ u$ tends to $\mathbf{y}_i$ as $s_i\to+\infty$ for $i=1,\dots,m$;}\\
            \text{$\pi_X\circ u$ tends to $\mathbf{y}_0$ as $s_0\to-\infty$;}\\
            \text{$\pi_{D_m}\circ u$ is a $\kappa$-fold branched cover of $D_m$.}
        \end{array}
    \right.
\end{align}
With the identification of the strip-like end $e_i$, $i=1,\dots,m$, with $[0,\infty)_{s_i}\times[0,1]_{t_i}$, the 3rd condition means that $u$ maps the neighborhoods of the punctures of $\mathbf{p}_i$ asymptotically to the Reeb chords $[0,1]_{t_i}\times \mathbf{y}_i$ as $s_i\to +\infty$. The 4th condition is similar. 

The $\mu^m$-composition map of $\mathcal{F}_\kappa(X)$ is then defined as
\begin{equation}
\label{eq-m_2}
    \mu^m(\mathbf{y}_1,\dots,\mathbf{y}_m)=\sum_{\mathbf{y}_0,\chi\leq\kappa}\#\mathcal{M}^{\mathrm{ind}=0,\chi}(\mathbf{y}_1,\dots,\mathbf{y}_m,\mathbf{y}_0)\cdot\hbar^{\kappa-\chi}\cdot\mathbf{y}_0,
\end{equation}
where the superscript $\chi$ denotes the Euler characteristic of $F$; the symbol $\#$ denotes the signed count of the corresponding moduli space.

\begin{theorem}
    \label{thm-ind}
    The Fredholm index of $\mathcal{M}^\chi(\mathbf{y}_1,\dots,\mathbf{y}_m,\mathbf{y}_0)$ is
    \begin{equation}
        \label{eq-ind}
        \mathrm{ind}(u)=(n-2)\chi+\mu+2\kappa-m\kappa n+m-2,
    \end{equation}
    where $\mu$ is the Maslov index of $u$, defined as in \cite[Section 4]{colin2020applications}.
\end{theorem}

If $2c_1(TX)=0$ and the Maslov classes of all involved Lagrangians vanish, then there exists a well-defined $\mathbb{Z}$-grading. In this case, the dimension of $\mathcal{M}^{\chi}(\mathbf{y}_1,\dots,\mathbf{y}_m,\mathbf{y}_0)$ can be rewritten as
\begin{equation}
    \label{eq-grading}
    \mathrm{ind}(u)=(n-2)(\chi-\kappa)+|\mathbf{y}_0|-|\mathbf{y}_1|-\dots-|\mathbf{y}_m|+m-2,
\end{equation}
where $|\mathbf{y}_i|=|y_{i1}|+\dots+|y_{i\kappa}|$. Note also that 
\begin{equation}
    \label{eq-hgrading}
    |\hbar|=2-n.
\end{equation}

We omit the details about the orientation of $\mathcal{M}(\mathbf{y}_1,\dots,\mathbf{y}_m,\mathbf{y}_0)$ and the $A_\infty$-relation, and refer the reader to \cite[Section 4]{colin2020applications}.

\subsection{Wrapped HDHF of disjoint cotangent fibers}
\label{subsection-wrap}

From now on, we restrict to the case $X=T^*{M}$, where $M$ is a compact manifold of dimension $n$. 
Let $\pi_M\colon T^*M\to M$ be the standard projection; by abuse of notation we also denote the projection map $\pi_{M}\circ \pi_{T^*{M}}$ simply by $\pi_{M}$.

The wrapped HDHF category is denoted by $\mathcal{W}_\kappa(T^*M)$.
Let $q_1,\dots,q_\kappa$ be $\kappa$ distinct 
points in a small disk $U\subset{M}$. We discuss the wrapped Heegaard Floer homology of the object $\sqcup_{i}T_{q_i}^*{M}$, where we assume $i$ to range from 1 to $\kappa$ in what follows.
\begin{remark}
    When we take $\sqcup_i T^*_{q_i} M$, the points $q_1,\dots, q_\kappa$ (as well as the Lagrangians $T^*_{q_1}M,\dots,T^*_{q_\kappa}M$) are ordered.
\end{remark}

Let $g$ be a Riemannian metric on ${M}$ and $|\cdot|$ be the induced norm on $T^*{M}$. Choose a time-dependent ``quadratic at infinity" Hamiltonian
\begin{gather}
\label{eq-H}
H_V\colon [0,1]\times T^*{M}\to\mathbb{R},\\
\nonumber    H_V(t,q,p)=\frac{1}{2}|p|^2+V(t,q),
\end{gather}
where $t\in[0,1]$, $q\in M$, $p\in T^*_q M$, and $V$ is some perturbation term with small $W^{1,2}$-norm in the $[0,1]\times M$-direction. The Hamiltonian vector field $X_{H_V}$ with respect to the canonical symplectic form $\omega=dq\wedge dp$ is then given by $i_{X_{H_V}}\omega=dH_V$. Let $\phi^t_{H_V}$ be the time-$t$ flow of $X_{H_V}$. 

By choosing $g$ and $V$ generically, we can guarantee that all Hamiltonian chords of $\phi^t_{H_V}$ between the cotangent fibers $\{T_{q_1}^*{M},\dots,T_{q_\kappa}^*{M}\}$ are nondegenerate. 

\begin{definition}
    \label{def-CW}
     The {\em wrapped Heegaard Floer chain complex} of $CW(\sqcup_{i}T_{q_i}^*{M})$ is $CF(\phi^1_{H_V}(\sqcup_{i}T_{q_i}^*{M}),\sqcup_{i}T_{q_i}^*{M})$.
\end{definition}

There is a subtlety when defining $A_\infty$-operations for wrapped HDHF and we briefly review the usual rescaling argument of \cite[Section 3]{abouzaid2010geometric}.

Start with a consistent collection $\{J_{D_m} ~|~ D_m\in \mathcal{A}_m,~ m\in \mathbb{Z}_{>0}\}$ as before.
We would like to impose Lagrangian boundary conditions $\tilde L_j:=\phi^{m-j}_{H_V}(\sqcup_{i}T_{q_i}^*{M})$ over the arc $\partial_j D_m$, but we need to make a modification near the end $e_0$ since the chain complex $CF(\phi^m_{H_V}(\sqcup_{i}T_{q_i}^*{M}),\sqcup_{i}T_{q_i}^*{M})$ is not naturally isomorphic to $CF(\phi^1_{H_V}(\sqcup_{i}T_{q_i}^*{M}),\sqcup_{i}T_{q_i}^*{M})$.

We will (slightly informally) explain how to modify the $\tilde L_j$ and the family ${J_{D_m}}$ in two steps.

\vskip.15in\noindent
{\em Step 1.} Let  $D_m \in \mathcal{A}_m \setminus N(\partial \mathcal{A}_m)$, where $N(\partial\mathcal{A}_m)$ is a small neighborhood of $\partial\mathcal{A}_m$ in $\overline{\mathcal{A}}_m$. 
Let $\psi^\rho$ be the time-$\mathrm{log}\,\rho$ flow of the Liouville vector field ${p}\partial_{{p}}$ of $(T^*{M}, {p}d{q})$. {(Note that because of the $\operatorname{log}$ term, $(\psi^\rho)^{-1}=\psi^{1/\rho}$.)} There is an isomorphism
\begin{align*}
    CF&(\psi^{\rho}(\phi^1_{H_V}(\sqcup_{i}T_{q_i}^*{M})),\psi^\rho(\sqcup_{i}T_{q_i}^*{M});\psi^\rho_*J_{D_m})\nonumber\\
    \cong CF&(\phi^1_{H_V}(\sqcup_{i}T_{q_i}^*{M}),\sqcup_{i}T_{q_i}^*{M};J_{D_m})\cong CW(\sqcup_{i}T_{q_i}^*{M}).
\end{align*}
Taking $\rho=1/m$, $\psi^{1/m}(\phi^1_{H_V}(\sqcup_{i}T_{q_i}^*{M}))$ limits to $\phi^m_{H_V}(\sqcup_{i}T_{q_i}^*{M})$ as $|{p}|\to\infty$ and $\psi^{1/m}(\sqcup_{i}T_{q_i}^*{M})$ remains the same as $\sqcup_{i}T_{q_i}^*{M}$.

We then modify $\tilde L_j$ only over the output end $e_0$ of $D_m$ to the totally real boundary condition $\tilde L'_j(D_m)$ so that for $s_0\leq -1$, $\tilde L'_0(D_m)$ is the trace of a fixed (but unspecified) exact, identity-at-infinity Lagrangian isotopy from $\phi^m_{H_V}(\sqcup_{i}T_{q_i}^*{M})$ to $\psi^{1/m}(\phi^1_{H_V}(\sqcup_{i}T_{q_i}^*{M}))$ as $s_0\to-\infty$ and $\tilde L'_m(D_m)$ is the identity trace from $\sqcup_{i}T_{q_i}^*{M}$ to $\psi^{1/m}(\sqcup_{i}T_{q_i}^*{M})$ as $s_0\to-\infty$. 

We convert the totally real boundary condition into a Lagrangian one by changing the symplectic form as follows:  First let $G(s_0,t_0,x)$ be an $(s_0,t_0)$-dependent, constant-at-infinity Hamiltonian function on $T^*M$ such that the Hamiltonian vector field $Y_{s_0,t_0}$ satisfying $i_{Y_{s_0,t_0}} \omega = d_{T^*M}G|_{(s_0,t_0)}$ (here $d_{T^*M}G$ is the component of $d$ in the $T^*M$-direction) induces the exact Lagrangian isotopies along $t_0=0,1$.  Then we replace $ds_0\wedge dt_0 +\omega$ by the symplectic form $\Omega= ds_0\wedge (dt_0-dG) +\omega$.  Along $t_0=0,1$, $\Omega=-ds_0\wedge dG +\omega= -ds_0\wedge d_{T^*M}G +\omega$ and the symplectic connection is given by $\partial_s + Y_{s_0,t_0}$.  Hence  $\tilde L'_0(D_m)$ and $\tilde L'_m(D_m)$ become Lagrangian with respect to $\Omega$.

Next 
we modify $J_{D_m}$ only over the output end $e_0$ of $D_m$ to an almost complex structure $J'_{D_m}$ so that for $s_0\leq -1$, $J'_{D_m}= \psi^{1/m}_*J_{D_m}$ and (J1)--(J3) still hold; this is possible since $\psi^\rho_* J_X\in \mathcal{J}_{X,\alpha}$ if $J_X\in \mathcal{J}_{X,\alpha}$. 

\vskip.15in\noindent
{\em Step 2.} We will briefly indicate how to inductively apply the modifications from Step 1 to construct consistent collections of totally real boundary conditions $\tilde L''_j(D_m)$ and almost complex structures $D_m\mapsto J''_{D_m}$, $D_m\in \mathcal{A}_m$ so that (J1)--(J3) still hold; see Figure~\ref{fig-Ainfty}.

\begin{figure}[ht]
	\begin{overpic}[scale=2.5]{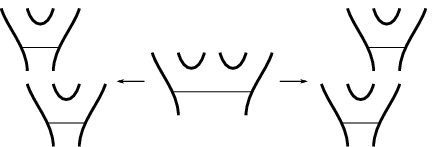}
		\put(33.8,18){\tiny $1$} \put(43.5,20.5){\tiny $\phi^1$} \put(53.3,20.5){\tiny $\phi^2$} \put(63.75,18){\tiny $\phi^3$}
		\put(47.5,15.7){\tiny $J_{D_3}$} \put(59.8,9){\tiny $\psi^{1/3}\phi^1$} \put(44.3,9){\tiny $\psi^{1/3} J_{D_3}$} \put(34,9){\tiny $\psi^{1/3}$}
		\put(48.5,25){\tiny \color{red} $\psi^{3}$}
		\put(4.7,11){\tiny $1$} \put(14.7, 13.5) {\tiny $\phi^1$} \put (24.7,11){\tiny $\phi^2$}

		\put(73.5,11){\tiny $1$} \put(83.5,13.5){\tiny $\phi^1$} \put(93.7,11){\tiny $\phi^2$}
		\put(13,8){\tiny $J_{D_2}$} \put(82,8){\tiny $J_{D_2}$} 
		\put(5.7,1){\tiny $\psi^{1/2}$} \put(14.9,1.5){\tiny $*$} \put(20.7,1){\tiny $\psi^{1/2}\phi^1$} \put(74.5,1){\tiny $\psi^{1/2}$} \put(83.7,1.5){\tiny $*$}\put(89,1){\tiny $\psi^{1/2}\phi^1$} 
		\put(9.1,19.5){\tiny $*$} \put(90.3,19.5){\tiny $*$}
		\put(1,5){\tiny \color{red} $\psi^{2}$} \put(95,5){\tiny \color{red} $\psi^{2}$}
		\put(-2.3,29){\tiny $1$} \put(8.3, 31.5){\tiny $\phi^1$} \put(18.3,29){\tiny $\phi^2$}
		\put(79.5,29){\tiny $1$} \put(89.5,31.5){\tiny $\phi^1$} \put(99.5,29){\tiny $\phi^2$}
		\put(7,25.8){\tiny $J_{D_2}$}\put(87.8,25.8){\tiny $J_{D_2}$}
		\put(43,0){\tiny $*=\psi^{1/2} J_{D_2}$}
		\put(-.5,18.3){\tiny $\psi^{1/2}$} \put(14.5, 18.3){\tiny $\psi^{1/2}\phi^1$} \put(80,18.3){\tiny $\psi^{1/2}$} \put(95,18.3){\tiny $\psi^{1/2}\phi^1$}
		\put(-5,25){\tiny \color{red} $\psi^4$} \put(100,25){\tiny \color{red} $\psi^{2}\phi^1 \psi^{2}$ }
	\end{overpic}
    \caption{Description of the totally real boundary conditions and almost complex structures for $a=b=2$. The middle is $D_3$ and the left and right are degenerations into $D_2$ and $D_2$. For a diffeomorphism $h$ of $T^*M$, the label $h$ is shorthand for $h(\sqcup_{i}T_{q_i}^*{M})$; we also abbreviate $\phi=\phi_{H_V}$.  A diffeomorphism in red means apply the diffeomorphism to all the labels on the closest disk.}
    \label{fig-Ainfty}
\end{figure}

For $D_m\in \mathcal{A}_m  \setminus N(\partial \mathcal{A}_m)$ we set $\tilde L''_j(D_m):= \psi^{m}(\tilde L'_j(D_m))$ and $J''_{D_m}:= (\psi^{m})_* J'_{D_m}$; this extra normalization is done for consistency with the evaluation map $\mathcal{F}$ in Section~\ref{section-f}.   

Suppose $D_m\in \mathcal{A}_m$ is close to breaking into $D_a$ and $D_b$, $a+b=m+1$, where the negative end of $D_a$ is glued to the positive end of $D_b$ corresponding to $p_{b-i}$, and $J_{D_m}$ is close to breaking into $J_{D_a}$ and $J_{D_b}$.  We then apply Step 1 to obtain $J'_{D_a}$ and $J'_{D_b}$, apply $\psi^{b}\phi^i_{H_V}\psi^{a}$ to $J'_{D_a}$ (resp.\ to $\tilde L'_j(D_a)$ for all $j$) and $\psi^{b}$ to $J'_{D_b}$ (resp.\ to $\tilde L'_j(D_b)$ for all $j$) so the ends match, and glue them together. 

{We then choose an interpolation between the above choices of boundary conditions and almost complex structures as in Figure~\ref{fig-Ainfty}.}

\vskip.15in
By abuse of notation, we write $J_{D_m}$ instead of $J''_{D_m}$ from now on.  Condition~\eqref{floer-condition} in the definition of the moduli space $\mathcal{M}({\bf y}_1,\dots,{\bf y}_m,{\bf y}_0)$ is modified as follows: 
\begin{itemize}
\item $\tilde L_{ij}$ is replaced by $\tilde L''_{ij}(D_m)$; and
\item the tuples ${\bf y}_i$ are replaced by the corresponding tuples of intersection points after applying the appropriate diffeomorphism to the two intersecting fiber Lagrangians; for example in the case where $D_m$ is close to breaking into $D_a$ and $D_b$, we apply $\psi^{b}$ to the intersection points corresponding to the positive ends of $D_b$.
\end{itemize}

\begin{lemma}
    \label{lemma-compactness}
    Fix a sufficiently generic consistent choice of almost complex structures as above.
    Let $d=0$ or $1$. Given $\mathbf{y}_1,\dots,\mathbf{y}_m\in CW(\sqcup_{i}T_{q_i}^*{M})$, $\mathcal{M}^{\mathrm{ind}=d,\chi}(\mathbf{y}_1,\dots,\mathbf{y}_m,\mathbf{y}_0)$ is empty for all but finitely many $\mathbf{y}_0$. If it is nonempty, $\mathcal{M}^{\mathrm{ind}=d,\chi}(\mathbf{y}_1,\dots,\mathbf{y}_m,\mathbf{y}_0)$ (and $\mathcal{M}^{\operatorname{ind}=d,\chi}(\mathbf{y}_1,\mathbf{y}_0)/\mathbb{R}$ if $m=1$) admits a compactification for each Euler characteristic $\chi$.
\end{lemma}
\begin{proof} 
This is similar to \cite[Appendix B]{abouzaid2010geometric} or \cite[Section 7]{abouzaid2010open}.
\end{proof}

Therefore, the $A_\infty$-operation
\begin{equation*}
    \mu^m\colon CW(\sqcup_{i}T_{q_i}^*{M})\otimes\dots\otimes CW(\sqcup_{i}T_{q_i}^*{M})\to CW(\sqcup_{i}T_{q_i}^*{M}),
\end{equation*}
can be defined using a sufficiently generic consistent collection, making the chain complex $CW(\sqcup_{i}T_{q_i}^*{M})$ into an $A_\infty$-algebra.

\begin{remark}
    The same prescription gives us the wrapped HDHF category $\mathcal{W}_\kappa(X)$ of a Liouville domain, which generalizes the wrapped Fukaya category $\mathcal{W}_1(X)$ of $X$.
\end{remark}

We now discuss the grading on $CW(\sqcup_{i}T_{q_i}^*{M})$. 
Since $c_1(T(T^*{M}))=0$ and the Maslov class of the Lagrangian vanishes, there is a well-defined $\mathbb{Z}$-grading. Following \cite[Section 1.3]{auroux2014beginner}, we choose a nonzero section $\mu$ of the trivial complex line bundle $\Lambda_\mathbb{C}^2 T^*(T^*{M})$: Let $\{U_\alpha\}$ be a cover of ${M}$. On any local chart $U_\alpha$ with coordinates $(x_\alpha^1,\dots,x_\alpha^n)$, we define 
\begin{equation}
\label{eq-mu-alpha}
    \mu_\alpha=(dx_\alpha^1-i\circ dx_\alpha^1\circ J)\wedge\dots\wedge(dx_\alpha^n-i\circ dx_\alpha^n\circ J),
\end{equation}
where $J$ is viewed as a bundle map which takes $v^*\in T^*M$ to its dual $v\in TM$ via $g$.
Let $\mu=\sum_\alpha \varphi_\alpha\mu_\alpha$, where $\{\varphi_\alpha\}$ is some partition of unity with respect to $\{U_\alpha\}$. Then we consider the phase function
\begin{gather}
    \varphi_\mu\colon {LGr}(T(T^*{M}))\to S^1\label{eq-phi-mu},\\
    A \mapsto\frac{\mu(v_1\wedge\dots\wedge v_n)^2}{||\mu(v_1\wedge\dots\wedge v_n)||^2},\nonumber
\end{gather}
where ${LGr}(T(T^*{M}))$ is the Grassmannian of Lagrangian planes in $T(T^*{M})$ and $\{v_1,\dots,v_n\}$ are tangent vectors that span $A$. Note that $\varphi_\mu$ is independent of the choice of $\{v_1,\dots,v_n\}$. For any loop $l$ in ${LGr}(T(T^*{M}))$, the Maslov index of $l$ is then defined as the degree of $\varphi_\mu$ on $l$. Moreover, if the Maslov class of some Lagrangian $L$ vanishes, we can lift $\varphi_\mu$ on $L$ to a grading function $\tilde{\varphi}_\mu\colon L\to \mathbb{R}$.

Given any $q\in{M}$, denote the tangent space of the zero section of $T^*M$ at $q$ by $T_q{M}$. 
One can check that $\varphi_\mu(T_q{M})=1\in S^1$ for any choice of $\mu$. 
Hence we can define a grading function on the zero section ${M}$ which is identically $0$. 
Similarly, given any $x\in T^*{M}$, denote the vertical space at $x$ by $V_x$. 
We check that $\varphi_\mu(V_x)=(-1)^{2n}=1\in S^1$ for any choice of $\mu$. 
Therefore, on any (unwrapped) cotangent fiber $T^*_q{M}$, we can lift $\varphi_\mu$ to a grading function $\tilde{\varphi}_\mu\colon T^*_q{M}\to\mathbb{R}$ which is identically $0$.
~\\

Now let $M=\Sigma$, a closed oriented surface of genus greater than 0. 
In this case $|\hbar|=2-n=0$ from (\ref{eq-hgrading}). Let $g$ be a Riemannian metric on ${\Sigma}$ which is a small perturbation of the flat metric when ${\Sigma}$ is a torus and of the hyperbolic metric when ${\Sigma}$ has genus greater than 1. In this case:

\begin{lemma}
    \label{lemma-grading}
    Fix $\mu$ by any choice satisfying (\ref{eq-mu-alpha}). Then the grading $|\mathbf{y}|=0$ for every $\mathbf{y}\in CF(\phi^1_{H_V}(\sqcup_{i}T_{q_i}^*\Sigma),\sqcup_{i}T_{q_i}^*\Sigma)$.
\end{lemma}
\begin{proof} Given $\mathbf{y}=\{y_1,\dots,y_\kappa\}\in CF(\phi^1_{H_V}(\sqcup_{i}T_{q_i}^*\Sigma),\sqcup_{i}T_{q_i}^*\Sigma)$, each $y_i\in \phi^1_{H_V}(L_i)\cap L_{i'}$ corresponds to a time-1 Hamiltonian chord from $L_i$ to $L_{i'}$, parametrized by $(q(t),p(t))$, $t\in[0,1]$. Its Legendre transform (see Definition \ref{def-leg-L}) gives a perturbed geodesic $\gamma$ on $\Sigma$. By classical results of Duistermaat \cite[Theorem 4.3]{duistermaat1976morse} and \cite[Section 1.2]{abbondandolo2006floer}, the Conley-Zehnder index of $y_i$ with respect to $\mu$ is equal to the Morse index of $\gamma$ with respect to its Lagrangian action. Lemma \ref{lemma-degree} then implies that $|y_i|=0$. Hence $|\mathbf{y}|=\sum\limits_{i=1}^{\kappa}|y_i|=0$.
\end{proof}

\begin{proposition}
\label{lemma-algebra}
The $A_\infty$-algebra $CW(\sqcup_{i}T_{q_i}^*\Sigma)$ is supported in degree zero, and hence is an ordinary algebra.
\end{proposition}
\begin{proof} 
The complex $CW(\sqcup_{i}T_{q_i}^*\Sigma)$ is supported in degree zero by Lemma \ref{lemma-grading} and $|\hbar|=0$. 
The $A_\infty$-operation $\mu^m=0$ for all $m\neq2$ since the degree of $\mu^m$ is $2-m$. Hence $\mu^2$ is the only nontrivial $A_\infty$-operation and $CW(\sqcup_{i}T_{q_i}^*\Sigma)$ is an ordinary algebra.
\end{proof}

\section{The relationship between the based loop space and the wrapped Floer homology of a cotangent fiber} 
\label{section-loop}

Let $M$ be a compact oriented manifold of dimension $n$. We review some basic properties of the based loop space on $M$ and in particular the relationship to the wrapped Floer homology of a single cotangent fiber of $T^*M$, i.e., the case when $\kappa=1$.
We refer the reader to \cite{abbondandolo2006floer,abouzaid2012wrapped} for more details.  

Let $g$ be a generic Riemannian metric on $M$ and let $\nabla$ be its associated Levi-Civita connection.

Consider the path space
\begin{equation*}
    {\Omega}(M,q_0,q_1)=\{\gamma\in C^0([0,1],M)~|~\gamma(0)=q_0,\gamma(1)=q_1\}.
\end{equation*}
There is a composition map which is simply the concatenation of paths:
\begin{gather*}
    {\Omega}(M,q_0,q_1)\times{\Omega}(M,q_1,q_2)\to{\Omega}(M,q_0,q_2),\\
    \gamma_1\gamma_2(t)=
    \left\{
        \begin{array}{lr}
            \gamma_1(2t),\,\,\,\,\,0\leq t\leq 1/2, &  \\
            \gamma_2(2t-1),\,\,\,\,1/2\leq t\leq 1. &  
        \end{array}
    \right.
\end{gather*}



In order to do Morse theory on the path space, we use $\Omega^{1,2}(M,q_0,q_1)$, the subset of $\Omega(M,q_0,q_1)$ consisting of paths in the class $W^{1,2}$.

There is a natural action functional on $\Omega^{1,2}(M,q_0,q_1)$.
Recall the Hamiltonian $H_V$, the Hamiltonian vector field $X_{H_V}$, and the time-$t$ Hamiltonian flow $\phi^t_{H_V}$ from Section \ref{subsection-wrap}.
Consider the function $L_V\colon [0,1]\times TM \to \mathbb{R}$ given by:
\begin{equation}
\label{eq-lagrangian}
    L_V(t,q,v)=\frac{1}{2}|v|^2-V(t,q),
\end{equation}
where $t\in[0,1]$, $q\in M$, and $v\in T_q M$. For each $\gamma\in\Omega^{1,2}(M,q_0,q_1)$, let
\begin{equation}
\label{eq-lagrangian-action}
    \mathcal{A}_V(\gamma)=\int^1_0 L_V(t,\gamma,\dot{\gamma})\,dt.
\end{equation}
It is well-known that $\mathcal{A}_V$ is a Morse function on $\Omega^{1,2}(M,q_0,q_1)$. Therefore we can define $CM_*(\Omega^{1,2}(M,q_0,q_1))$ as the Morse complex generated by the critical points of $\mathcal{A}_V$ and with differential induced by $\mathcal{A}_V$ and the metric $g$. We omit the details which can be found in \cite[Section 2]{abbondandolo2006floer} and simply denote the Morse homology group of $CM_*(\Omega^{1,2}(M,q_0,q_1))$ by $HM_*(\Omega^{1,2}(M,q_0,q_1))$.

\begin{definition}
    \label{def-v-geodesic}
    A {\em $V$-perturbed geodesic $\gamma$} is a map $[0,1]\to M$ such that
    \begin{equation}
        \label{eq-v-geodesic}
        \nabla_{\dot{\gamma}}\dot{\gamma}=-\nabla V,
    \end{equation}
    where $\nabla V$ denotes the gradient of $V$ with respect to $g$. 
\end{definition}

By a standard calculus of variations computation, the critical points of $\mathcal{A}_V$ on $\Omega^{1,2}(M,q_0,q_1)$ are exactly $V$-perturbed geodesics.
~\\

We now recall some standard facts following \cite[Section 2.1]{abbondandolo2006floer}.
Let $\mathcal{C}_{H_V}$ be the set of time-1 integral curves of $X_{H_V}$ on $T^*M$ and let $\mathcal{C}_{L_V}$ be the set of time-1 $V$-perturbed geodesics on $M$, i.e.,
\begin{align*}
    \mathcal{C}_{H_V}&\coloneqq\{\zeta:[0,1]\to T^*M~|~\zeta(t)=\phi^t_{H_V}\circ\zeta(0)\},\\
    \mathcal{C}_{L_V}&\coloneqq\{\gamma:[0,1]\to M~|~\nabla_{\dot{\gamma}}\dot{\gamma}=-\nabla V\}.
\end{align*}

We define a map $\mathcal{L}\colon \mathcal{C}_{H_V} \to \mathcal{C}_{L_V}$ as follows:  Given $\zeta\in \mathcal{C}_{H_V}$, let $\mathcal{L}(\zeta)$ be the path $[0,1]\to M$ given by
\begin{equation*}
    \mathcal{L}(\zeta)(t)\coloneqq \pi_M\circ\zeta(t).
\end{equation*}
One can verify that $\mathcal{L}(\zeta)$ satisfies (\ref{eq-v-geodesic}) and hence belongs to $\mathcal{C}_{L_V}$.

We define the inverse map $\mathcal{L}^{-1}\colon \mathcal{C}_{L_V}  \to \mathcal{C}_{H_V}$ as follows: 
Given $\gamma\in\mathcal{C}_{L_V}$, we define $\mathcal{L}^{-1}(\gamma)\colon[0,1]\to T^*M$ as
\begin{equation}
    \label{eq-leg-L}
    \mathcal{L}^{-1}(\gamma)\coloneqq (\gamma(t),dL(t,\gamma(t),\dot{\gamma}(t))|_{T^v_{(\gamma(t),\dot{\gamma}(t))}TM}),
\end{equation}
where $T^v_{(\gamma(t),\dot{\gamma}(t))}TM$ is the vertical fiber $\ker D\pi_M\cong T_{\gamma(t)}M$ at $(\gamma(t),\dot{\gamma}(t))\in TM$ and $\pi_M\colon TM\to M$ is the projection. 

\begin{definition}
    \label{def-leg-L}
    We call $\mathcal{L}$ the {\em Legendre transform} and call $\mathcal{L}^{-1}$ the {\em inverse Legendre transform}. 
\end{definition}

\begin{remark}
    \label{rmk-leg-L}
    Each generator $y\in CF(\phi^1_{H_V}(T_{q_0}^*{M}),T_{q_1}^*{M})$ corresponds to a time-1 integral curve of $X_{H_V}$ from $T_{q_0}^*{M}$ to $T_{q_1}^*{M}$:
    \begin{equation}
        l_y\colon[0,1]\to T^*M,\quad l_y(t)=\phi^{t-1}_{H_V}(y).
    \end{equation}
    We define $\mathcal{L}(y)$ to be $\mathcal{L}(l_y)$. Conversely, $\mathcal{L}^{-1}$ maps time-1 $V$-perturbed geodesics from $q_0$ to $q_1$ to a generator of $CF(\phi^1_{H_V}(T_{q_0}^*{M}),T_{q_1}^*{M})$.
\end{remark}

\vskip.1in
Since the inclusion
\begin{equation}
    \Omega^{1,2}(M,q_0,q_1)\hookrightarrow\Omega(M,q_0,q_1)
\end{equation}
is a homotopy equivalence, we deduce that $HM_*(\Omega^{1,2}(M,q_0,q_1))$ is isomorphic to the singular homology group $H_*(\Omega(M,q_0,q_1))$.

When $q_0=q_1=q$, we get the based loop space
\begin{equation*}
    \Omega(M,q)=\{\gamma\in C^0([0,1],M)~|~\gamma(0)=\gamma(1)=q\}.
\end{equation*}

\begin{theorem}[Theorem B of \cite{abbondandolo2010floer}]
\label{thm-ab}
    There is an isomorphism of graded algebras:
    \begin{equation}
    \label{eq-ab}
      H_{-*}(\Omega(M,q)) \to HW^*(T_q^*M).
    \end{equation}
\end{theorem}

In the opposite direction, Abouzaid constructed a chain level evaluation map $CW^*(T_q^*M) \to C_{-*}(\Omega(M,q))$. 
It induces an isomorphism on the level of homology:
\begin{equation}
\label{eq-F-Ab}
    \tilde{\mathcal{F}}\colon HW^*(T_q^*M) \to H_{-*}(\Omega(M,q)).
\end{equation}

\begin{remark}
    \label{rmk-ab}
    Theorem \ref{thm-ab} also holds for path spaces by \cite{abbondandolo2006floer}, i.e., there is an isomorphism
    \begin{equation}
    \label{eq-ab-ends}
        H_{-*}(\Omega(M,q_0,q_1))\to HW^*(T_{q_0}^*M,T_{q_1}^*M)
    \end{equation}
    where the right-hand side is the wrapped Floer homology group whose generators are time-1 Hamiltonian flows of $\phi^t_{H_V}$ from $T_{q_0}^*M$ to $T_{q_1}^*M$.
\end{remark}

Let us now specialize to $M=\Sigma$, a closed oriented surface of genus greater than 0.
In this case we further see:
\begin{lemma}
\label{lemma-degree}
    $H_*(\Omega(\Sigma,q_0,q_1))$ is supported in degree 0.
\end{lemma}
\begin{proof} Recall that $V$ is small in the $W^{1,2}$-norm. If $\Sigma$ is a torus, then we can assume that $g$ is the flat metric, where all $V$-perturbed geodesics with $V$ sufficiently small are minimal and isolated. If the genus of $\Sigma$ is greater than 1, then we can assume that $g$ is the hyperbolic metric with constant curvature $-1$. It is well known that on a hyperbolic surface, there is a unique $V$-perturbed geodesic in each homotopy class of paths with fixed endpoints for $V$ sufficiently small. For details the reader is referred to Milnor \cite[Lemma 19.1]{milnor2016morse}.
Hence the Morse indices of all critical points of $\Omega^{1,2}(\Sigma,q_0,q_1)$ are 0.
\end{proof}

Note that when $M=\Sigma$, all terms of (\ref{eq-ab}) vanish except for $*=0$ by Lemma \ref{lemma-degree}. We write $HW(T_q^*\Sigma)$ for $HW^0(T_q^*\Sigma)$. 
Moreover, $H_{0}(\Omega(\Sigma,q))$ is isomorphic to the group algebra $\mathbb{Z}[\pi_1(\Sigma,q)]$ of the fundamental group $\pi_1(\Sigma,q)$.


\section{Hecke algebras} \label{section-Hecke}

Theorem (\ref{thm-ab}) and the isomorphism (\ref{eq-F-Ab}) relate $H_0(\Omega(\Sigma,q))$ to the wrapped
Floer homology $HW(T_q^*\Sigma)$ of a single cotangent fiber, and represent the special case
of $\kappa=1$ in HDHF. 
In this section, we discuss the generalization of $H_0(\Omega(\Sigma,q))$ to $\kappa \geq 1$, which we show to be equivalent to $HW(\sqcup_i T^*_{q_i}\Sigma)$ in later sections.
\vskip.1in
\noindent{\em Summary.} Consider the based loop space of the unordered configuration space of $\kappa$ points on $\Sigma$. 
Its $0$th homology is isomorphic to the group algebra of the braid group of $\Sigma$.   
The {\em braid skein algebra} $\mathrm{BSk}_\kappa(\Sigma,{\bf q})$, due to Morton and Samuelson \cite[Definition 3.1]{morton2021dahas}, is a quotient of the group algebra of the braid group by the HOMFLY skein relation and the marked point relation. Here {\bf q} is a $\kappa$-tuple of distinct points on $\Sigma$.
By reformulating the marked point relation, we obtain the {\em surface Hecke algebra} $\mathrm{H}_\kappa(\Sigma,{\bf q})$ in Definition \ref{def-hecke}. 
This surface Hecke algebra serves as an intermediary between the wrapped Floer homology and the braid skein algebra.
On one hand, we show that $\mathrm{H}_\kappa(\Sigma,{\bf q})$ and $\mathrm{BSk}_\kappa(\Sigma)$ are isomorphic up to a change of variables in Proposition \ref{prop-BSkHecke}.
On the other hand, we will construct an evaluation map from the wrapped Floer homology to $\mathrm{H}_\kappa(\Sigma,{\bf q})$ in Section \ref{subsection-ev}.

Note that the braid skein algebra for $\Sigma=T^2$ is isomorphic to the double affine Hecke algebra (DAHA) of $\mathfrak{gl}_{\kappa}$ \cite[Theorem 3.7]{morton2021dahas}. Hence $\mathrm{H}_\kappa(T^2,{\bf q})$ is also isomorphic to the DAHA.  


\vspace{.2cm}
Let $\mathrm{UConf}_{\kappa}(\Sigma)=\{\{q_1,\dots,q_{\kappa}\}~|~q_i\in \Sigma, ~q_i\neq q_j ~\mbox{for}~i\neq j\}$ be the configuration space of $\kappa$ unordered points on $\Sigma$. 
Fix a basepoint $\mathbf{q} \in \mathrm{UConf}_{\kappa}(\Sigma)$.
The based loop space $\Omega(\mathrm{UConf}_{\kappa}(\Sigma),\mathbf{q})$ consists of $\kappa$-strand braids in $\Sigma$.
Note that $H_0(\Omega(\mathrm{UConf}_{\kappa}(\Sigma),\mathbf{q}))$ is isomorphic to the group algebra $\mathbb{Z}[\mathrm{Br}_{\kappa}(\Sigma,\mathbf{q})]$, where $\mathrm{Br}_{\kappa}(\Sigma,\mathbf{q})$ denotes the braid group of $\Sigma$.

Fix a marked point $\star \in \Sigma$ which is disjoint from $\mathbf{q}$. 
Let $\mathrm{Br}_{\kappa,1}(\Sigma,\mathbf{q},\star)$ be the subgroup of $\mathrm{Br}_{\kappa+1}(\Sigma,\mathbf{q}\sqcup\{\star\})$ consisting of braids whose last strand connects $\star$ to itself by a straight line in $[0,1]\times\Sigma$.

\begin{definition}[Morton-Samuelson]
    \label{def-skein}
    The braid skein algebra $\mathrm{BSk}_\kappa(\Sigma,\mathbf{q})$ is the quotient of the group algebra $\mathbb{Z}[s^{\pm1}, c^{\pm1}][\mathrm{Br}_{\kappa,1}(\Sigma,\mathbf{q},\star)]$ by two local relations:
\begin{enumerate}
\item the HOMFLY skein relation 
    \begin{equation}
        \label{eq-skein'}
        \includegraphics[width=1cm,valign=c]{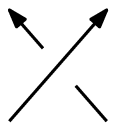}-\includegraphics[width=1cm,valign=c]{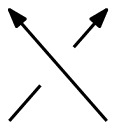}=(s-s^{-1})\includegraphics[width=1cm,valign=c]{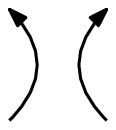},
    \end{equation}
\item the marked point relation $P=c^2$
    \begin{equation}
        \label{eq-c}
        P:=\includegraphics[height=2cm,valign=c]{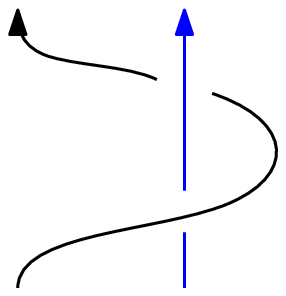}\,=\,c^2\,\includegraphics[height=2cm,valign=c]{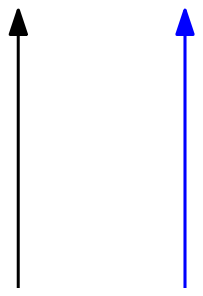}.
    \end{equation}
\end{enumerate}
Here the black lines are strands between basepoints in $\mathbf{q}$ and the straight blue line connects the marked point $\star$ to itself. 

The product is given by the concatenation of braids. 
\end{definition}
 
For the purposes of relating $\mathrm{BSk}_\kappa(\Sigma,{\bf q})$ to the wrapped Floer homology, it is more convenient to give another description of the marked point relation (\ref{eq-c}). 
Let 
\begin{align}
    \Omega&(\mathrm{UConf}_{\kappa}(\Sigma\setminus\{\star\}),\mathbf{q})\label{eq-confstar}\\
    &=\{\gamma\in C^0([0,1],\mathrm{UConf}_{\kappa}(\Sigma \setminus \{\star\})) ~|~ \gamma(0)=\gamma(1)={\bf q}\}.\nonumber
\end{align}
Given $\gamma_1,\gamma_2\in\Omega(\mathrm{UConf}_{\kappa}(\Sigma\setminus\{\star\}),\mathbf{q})$ viewed as based loops on $\mathrm{UConf}_{\kappa}(\Sigma)$, let $H\colon [0,1]^2 \to \mathrm{UConf}_{\kappa}(\Sigma)$ be a homotopy between $\gamma_1$ and $\gamma_2$ relative to the boundary, i.e.,
$$H(t,0)=\gamma_1(t),\quad H(t,1)=\gamma_2(t),\quad  H(0,s)=H(1,s)={\bf q}.$$

The homotopy $H$ may intersect the marked point $\star$.  We define $\langle H,\star\rangle\coloneqq\langle H,Y\rangle$, the algebraic intersection number of $H$ and $Y$, where
$$Y:=\{\{p_1,\dots,p_{\kappa}\} \in \mathrm{UConf}_{\kappa}(\Sigma)~|~p_i=\star ~\mbox{for some}~ i\}$$
is a codimension two submanifold of $\mathrm{UConf}_{\kappa}(\Sigma)$.
Here the orientation of $Y$ is induced from that of $\Sigma$ and the orientation of $H$ is induced from $-dt\wedge ds$ on the square $[0,1]^2$ (note the coordinates on $[0,1]^2$ are $t,s$ in that order). 
This is well-defined since $H(\partial([0,1]^2))\cap Y=\varnothing$.

We identify $H_0(\Omega(\mathrm{UConf}_{\kappa}(\Sigma\setminus\{\star\}),\mathbf{q}))$ with the group algebra $\mathbb{Z}[\mathrm{Br}_{\kappa}(\Sigma\setminus\{\star\},\mathbf{q})]$ of the braid group $\mathrm{Br}_{\kappa}(\Sigma\setminus\{\star\},\mathbf{q})$.

\begin{definition}
    \label{def-brc}
    The $c$-deformed braid group $\mathrm{Br}_{\kappa}(\Sigma,\mathbf{q})_c$ of $\Sigma$ is generated by $\mathrm{Br}_{\kappa}(\Sigma\setminus\{\star\},\mathbf{q})$ and a central element $c$, subject to the following $c$-deformed homotopy relation:
\begin{equation}
        \label{eq-chomotopy}
[\gamma_2]=c^{2\langle H,\star\rangle}[\gamma_1],
\end{equation}
    where $\gamma_i \in \Omega(\mathrm{UConf}_{\kappa}(\Sigma\setminus\{\star\}),\mathbf{q})$, $H$ is the homotopy between them as above, and $[\gamma_i] \in H_0(\Omega(\mathrm{UConf}_{\kappa}(\Sigma\setminus\{\star\}),\mathbf{q})) \cong \mathbb{Z}[\mathrm{Br}_{\kappa}(\Sigma\setminus\{\star\},\mathbf{q})]$.  
\end{definition}

\begin{remark}
\label{rmk-brc}
The group algebra $\mathbb{Z}[\mathrm{Br}_{\kappa}(\Sigma,\mathbf{q})_c]$ is naturally isomorphic to the quotient of $C_0(\Omega(\mathrm{UConf}_{\kappa}(\Sigma\setminus\{\star\}),\mathbf{q}))\otimes \mathbb{Z}[c^{\pm1}]$ by the $c$-deformed homotopy relation (\ref{eq-chomotopy}).
\end{remark}

See Figure \ref{fig-c-homotopy} for an example when $\kappa=1$ and $\langle H,Y\rangle =1$.

\begin{figure}[ht]
    \centering
    \includegraphics[width=4cm]{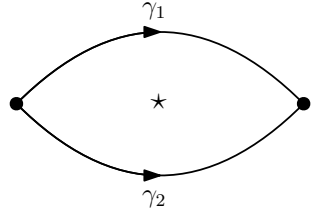}
    \caption{The deformed homotopy relation $[\gamma_2]=c^{2}[\gamma_1]$. Both dotted points are $q \in \Sigma$.}
    \label{fig-c-homotopy}
\end{figure}

Specializing to $c=1$, $\mathrm{Br}_{\kappa}(\Sigma,\mathbf{q})_c$ recovers $\mathrm{Br}_{\kappa}(\Sigma,\mathbf{q})$. 
Hence there is a central extension of groups:
\begin{equation}
        \label{eq-centralbrc}
1 \to \langle c \rangle \to \mathrm{Br}_{\kappa}(\Sigma,\mathbf{q})_c \to \mathrm{Br}_{\kappa}(\Sigma,\mathbf{q}) \to 1,
\end{equation}
where $\langle c \rangle$ is the free abelian group generated by $c$.

Let $\langle \mathrm{Br}_{\kappa,1}(\Sigma,\mathbf{q},\star),c \rangle$ denote the group generated by $\mathrm{Br}_{\kappa,1}(\Sigma,\mathbf{q},\star)$ and a central element $c$. 
Define the group $\mathrm{Br}_{\kappa}(\Sigma,\mathbf{q})'_c$ as the quotient of $\langle \mathrm{Br}_{\kappa,1}(\Sigma,\mathbf{q},\star),c \rangle$ by the marked point relation $P=c^2$ (\ref{eq-c}). 
There are two natural projections from $\langle \mathrm{Br}_{\kappa,1}(\Sigma,\mathbf{q},\star),c \rangle$ to $\mathrm{Br}_{\kappa}(\Sigma,\mathbf{q})'_c$ and $\mathrm{Br}_{\kappa}(\Sigma,\mathbf{q})_c$, respectively:
$$\xymatrix{
& \langle \mathrm{Br}_{\kappa,1}(\Sigma,\mathbf{q},\star),c \rangle \ar[dr] \ar_{P=c^2}[dl]& \\
\mathrm{Br}_{\kappa}(\Sigma,\mathbf{q})'_c \ar[rr]^{\phi}& &\mathrm{Br}_{\kappa}(\Sigma,\mathbf{q})_c
}$$
The relation $P=c^2$ also holds in $\mathrm{Br}_{\kappa}(\Sigma,\mathbf{q})_c$.
Therefore, there is a natural induced map $\phi\colon \mathrm{Br}_{\kappa}(\Sigma,\mathbf{q})'_c \to \mathrm{Br}_{\kappa}(\Sigma,\mathbf{q})_c$. 

The algebra $\mathrm{Br}_{\kappa}(\Sigma,\mathbf{q})'_c$ also fits into a central extension similar to (\ref{eq-centralbrc}):
$$1 \to \langle c \rangle \to \mathrm{Br}_{\kappa}(\Sigma,\mathbf{q})'_c \to \mathrm{Br}_{\kappa}(\Sigma,\mathbf{q}) \to 1.$$
Moreover, the following diagram commutes:
$$\xymatrix{
1 \ar[r] & \langle c \rangle \ar[r] & \mathrm{Br}_{\kappa}(\Sigma,\mathbf{q})_c \ar[r] & \mathrm{Br}_{\kappa}(\Sigma,\mathbf{q}) \ar[r] & 1 \\
1 \ar[r] & \langle c \rangle \ar[r] \ar@{=}[u] & \mathrm{Br}_{\kappa}(\Sigma,\mathbf{q})'_c \ar[r] \ar[u]_{\phi}& \mathrm{Br}_{\kappa}(\Sigma,\mathbf{q}) \ar@{=}[u] \ar[r] & 1
.}$$
By the Five Lemma, $\phi$ is an isomorphism.

Recall from Definition \ref{def-skein} that the braid skein algebra $\mathrm{BSk}_\kappa(\Sigma,\mathbf{q})$ is the quotient of $\mathbb{Z}[s^{\pm1}][\mathrm{Br}_{\kappa}(\Sigma,\mathbf{q})'_c]$ by the skein relation (\ref{eq-skein'}). 
We introduce the variable $\hbar=s-s^{-1}$ and add the corresponding skein relation to $\mathrm{Br}_{\kappa}(\Sigma,\mathbf{q})_c$.


\begin{definition}
    \label{def-hecke}
    The Hecke algebra $\mathrm{H}_{\kappa}(\Sigma,\mathbf{q})$ of $\Sigma$ is the quotient of the group algebra $\mathbb{Z}[\hbar][\mathrm{Br}_{\kappa}(\Sigma,\mathbf{q})_c]$ by the local skein relation:
        \begin{equation}
            \label{eq-skein}
            \includegraphics[width=1cm,valign=c]{poscross.eps}-\includegraphics[width=1cm,valign=c]{negcross.eps}=\hbar\includegraphics[width=1cm,valign=c]{node.eps}.
        \end{equation}  
\end{definition}

%
%
%
%
%
 
We now apply a change of variables $\hbar=s-s^{-1}$ to $\mathrm{BSk}_{\kappa}(\Sigma,\mathbf{q})$ and change the coefficient ring from $\mathbb{Z}[s^{\pm1},c^{\pm1}]$ to $\mathbb{Z}[\hbar,c^{\pm1}]$; the resulting algebra will be denoted by $\mathrm{BSk}_{\kappa}(\Sigma,\mathbf{q})|_{\hbar}$. 
 
\begin{proposition}
\label{prop-BSkHecke}
The algebra $\mathrm{H}_{\kappa}(\Sigma,\mathbf{q})$ is naturally isomorphic to $\mathrm{BSk}_{\kappa}(\Sigma,\mathbf{q})|_{\hbar}$.
\end{proposition}


%
%
\vspace{.2cm}
Finally, we consider the degeneration $\hbar=0$ for later use. 
The skein relation (\ref{eq-skein}) then reduces to the symmetric group relation. 
Further setting $c=1$, the free abelian group $\mathrm{H}_{\kappa}(\Sigma,\mathbf{q})|_{\hbar=0,c=1}$ has a $\mathbb{Z}$-basis 
$$\mathfrak{B}:=\{\gamma_1\cdots \gamma_{\kappa}\cdot \sigma~|~ \gamma_i \in \pi_1(\Sigma,q_i), \sigma \in S_{\kappa}\}.$$ 
Here $\gamma_i$ is the homotopy class of the $i$th strand in the braid group and $\sigma$ is the image of the natural quotient map from the braid group to the symmetric group. 
Hence we have a natural isomorphism of algebras 
\begin{align*}
\mathrm{H}_{\kappa}(\Sigma,\mathbf{q})|_{\hbar=0,c=1} \cong \mathbb{Z}[\prod_i \pi_1(\Sigma,q_i) \rtimes S_{\kappa}],
\end{align*}
where the symmetric group $S_{\kappa}$ acts on $\prod_i \pi_1(\Sigma,q_i)$ by permuting the factors.
Adding the parameter $c$, let $\pi_1(\Sigma,q)_c$ denote $\mathrm{Br}_{1}(\Sigma,q)_c$, the $1$-strand $c$-deformed braid group. 
Its group algebra $\mathbb{Z}[\pi_1(\Sigma,q)_c]$ is a $\mathbb{Z}[c^{\pm1}]$-algebra. 
\begin{lemma}
    \label{lemma-dim-Heckeh=0}
    There is an isomorphism of $\mathbb{Z}[c^{\pm1}]$-algebras:
    \begin{equation}
            \label{eq-bskh=0}
    \mathrm{H}_{\kappa}(\Sigma,\mathbf{q})|_{\hbar=0} \cong \left(\otimes_i\,\mathbb{Z}[\pi_1(\Sigma,q_i)_c]\right) \rtimes S_{\kappa},
    \end{equation} 
    where the tensor product is over $\mathbb{Z}[c^{\pm1}]$.
\end{lemma}
\begin{proof}
The free abelian group $\mathrm{H}_{\kappa}(\Sigma,\mathbf{q})|_{\hbar=0}$ has $\mathbb{Z}$-basis $$\mathfrak{B}_c:=\{\gamma_1\cdots \gamma_{\kappa}\cdot \sigma~|~ \gamma_i \in \pi_1(\Sigma,q_i)_c, \sigma \in S_{\kappa}\}.$$
We have a natural surjective homomorphism 
\begin{align*}
\left(\tilde{\otimes}_i\,\mathbb{Z}[\pi_1(\Sigma,q_i)_c]\right) \rtimes S_{\kappa} \to \mathrm{H}_{\kappa}(\Sigma,\mathbf{q})|_{\hbar=0},
\end{align*} 
where $\tilde{\otimes}$ denotes the tensor product over $\mathbb{Z}$. 
It factors through 
$$\left(\otimes_i\,\mathbb{Z}[\pi_1(\Sigma,q_i)_c]\right) \rtimes S_{\kappa} \to \mathrm{H}_{\kappa}(\Sigma,\mathbf{q})|_{\hbar=0}$$
since $c$ is central in $\mathrm{H}_{\kappa}(\Sigma,\mathbf{q})|_{\hbar=0}$. 
The latter map is injective since the domain and target have the same $\mathbb{Z}[c^{\pm1}]$-basis $\mathfrak{B}$. 
\end{proof}

\section{The parameter $c$ in HDHF}
\label{section-c}
Motivated by the $c$-deformed homotopy relation (\ref{eq-chomotopy}), in this section we define $CW(\sqcup_{i}T_{q_i}^*\Sigma)_c$, the wrapped HDHF homology of disjoint cotangent fibers with an additional parameter $c$. 
We inherit the notation from Section \ref{subsection-wrap}. 

Let $CF(\phi^1_{H_V}(\sqcup_{i}T_{q_i}^*\Sigma),\sqcup_{i}T_{q_i}^*\Sigma)_c\coloneqq CF(\phi^1_{H_V}(\sqcup_{i}T_{q_i}^*\Sigma),\sqcup_{i}T_{q_i}^*\Sigma)\otimes \mathbb{Z}[c^{\pm 1}]$ as a $\mathbb{Z}$-module.


Given $u\in\mathcal{M}(\mathbf{y}_1,\dots,\mathbf{y}_m,\mathbf{y}_0)$ of index 0 or 1, consider its projection to $\Sigma$ and denote the image by $\pi_\Sigma(u)$. 
We enhance the $A_\infty$-operations from Section \ref{subsection-wrap} to include $c$-coefficients by keeping track of modified intersections of $\pi_\Sigma (u)$ and a fixed marked point $\star\in\Sigma$. 
Note that we cannot directly take the intersection number of $\pi_\Sigma(u)$ and $\star$ since the boundary of $\pi_\Sigma(u)$ could cross $\star$ in a generic 1-parameter family and the $A_\infty$-relation would not be satisfied. 

To remedy this issue, we carefully choose the marked point $\star\in\Sigma$ and modify $\pi_\Sigma(u)$. 

Consider the set of $V$-perturbed geodesics (see Definition \ref{def-v-geodesic}) with endpoints in $\{q_1,\dots,q_\kappa\}$. Since there are only countably many such $V$-perturbed geodesics, we can choose a generic marked point $\star\in\Sigma$ in the complement of the images of these perturbed geodesics. 

We then homotop the boundary of $\pi_\Sigma(u)$ on $\Sigma$ to piecewise $V$-perturbed geodesics, which can be guaranteed to be disjoint from $\star$; see Figure \ref{fig-c-modify}.
\begin{figure}[ht]
    \centering
    \includegraphics[width=6cm]{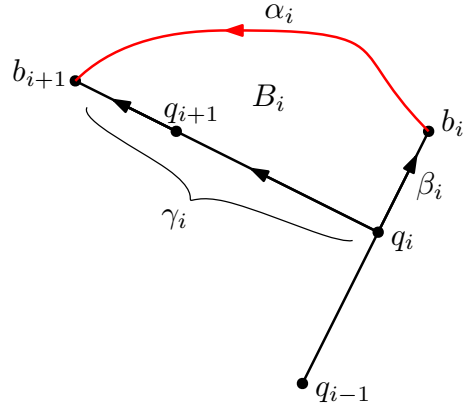}
    \caption{The projection $\pi_\Sigma(u)$ of $u$ to $\Sigma$. The path $\alpha_i$ is the projection of $u|_{\partial_i\dot F}$, where  $\partial_i\dot{F}$ is a component of $\partial \dot{F}$. We choose any homotopy from $\alpha_i$ to $-\beta_i+\gamma_i$, where the relative homology class $B_i$ that is swept out does not depend on the choice of homotopy. In the figure we are assuming that $V$ is independent of $t$ and hence $q_i$, $q_{i+1}$ and $b_{i+1}$ lie on the same $V$-perturbed geodesic.} 
    \label{fig-c-modify}
\end{figure}
Specifically, suppose the domain of $u$ is $\dot F$ and $\{p_{i-1},p_i,p_{i+1}\}$ are three consecutive boundary marked points on $\partial F$, ordered according to the boundary orientation of $\partial F$.
Write $\pi_{T^*\Sigma} \circ u(p_i)$ etc. for the values of the continuous extensions of $\pi_{T^*\Sigma} \circ u$ to the puncture $p_i$ etc.
Denote the boundary arc from $p_i$ to $p_{i+1}$ by $\partial_i\dot{F}$. 
Suppose the cotangent fibers $T^*_{q_{i-1}}\Sigma$, $T^*_{q_{i}}\Sigma$ and $T^*_{q_{i+1}}\Sigma$ are wrapped using the wrapping functions $\psi_{i-1}$, $\psi_{i}$ and $\psi_{i+1}$, respectively. 
For $r=0,1$, suppose $\pi_{T^*\Sigma}\circ u(p_{i+r})\in \psi_{i+r-1}(T^*_{q_{i+r-1}}\Sigma)\cap\psi_{i+r}(T^*_{q_{i+r}}\Sigma)$, which corresponds to the intersection of two Hamiltonian chords that start from $T^*_{q_{i+r-1}}\Sigma$ (induced by $\psi_{i+r-1}$) and $T^*_{q_{i+r}}\Sigma$ (induced by $\psi_{i+r}$). The Legendre transforms of these Hamiltonian chords (see Definition \ref{def-leg-L}) correspond to certain $V$-perturbed geodesics on $\Sigma$: Let $\beta_{i+r-1}$ be that from $q_{i+r-1}$ to $b_{i+r-1}$, and $\gamma_{i+r-1}$ be that from $q_{i+r-1}$ to $b_{i+r}$, where $b_i=\pi_{\Sigma}\circ u(p_{i})$.  (Observe that $q_i$, $q_{i+1}$ and $b_{i+1}$ lie on the same $V$-perturbed geodesic, if $V=V(t,x)$ is independent of $t$; this is because being able to intersect in $T^*M$ means that the tangent vector of the two $V$-perturbed geodesics at $b_i$ are the same.)
Let $\alpha_i$ be the path $\pi_{\Sigma}\circ u(\partial_i \dot{F})$. 

Fix parametrizations of $\alpha_i$, $\beta_i$ and $\gamma_i$ by the interval $[0,1]$. 
Note that $\alpha_i$ and its {\em piecewise $V$-geodesic replacement $\beta_i^{-1}\cdot\gamma_i$} are homotopic as paths from $b_i$ to $b_{i+1}$; this is due to the facts that $\alpha_i$, $\beta_i$ and $\gamma_i$ are all projections of paths in $\psi_i(T^*_{q_{i}}\Sigma)$ and that $T^*_{q_{i}}\Sigma$ is contractible. 
Let $B_i$ be a homotopy between $\alpha_i$ and $\beta_i^{-1}\cdot\gamma_i$ relative to boundary.
We extend the image $\pi_\Sigma(u)$ on $\Sigma$ by the homotopy $B_i$ for all $\partial_i\dot{F}\subset\partial\dot{F}$ so that its new boundary lies in $C\coloneqq\bigcup_i(\beta_i\cup\gamma_i)$; this defines a relative homology class $[\pi_\Sigma(u)]'\in H_2(\Sigma,C)$.

Given two homotopies $B_i$ and $B_i'$ from $\alpha_i$ to its piecewise $V$-geodesic replacement $\beta_i^{-1}\cdot\gamma_i$, their difference determines a map $S^2\to\Sigma$, which induces the zero map on $H_2(S^2)\to H_2(\Sigma,C)$. Hence $[\pi_\Sigma(u)]'$ does not depend on the choice of $\{B_i\}$ and the algebraic intersection number $\langle u,\star\rangle\coloneqq\langle [\pi_\Sigma(u)]',\star\rangle$
is well-defined.



We modify the $\mu^m$-composition map so that 
\begin{equation}
\label{eq-mu-2-c}
    \mu^m(\mathbf{y}_1,\dots,\mathbf{y}_m)=\sum_{u\in\mathcal{M}^{\mathrm{ind}=0}(\mathbf{y}_1,\dots,\mathbf{y}_m,\mathbf{y}_0)} (-1)^{\natural(u)}\cdot c^{2\langle u,\star\rangle}\cdot\hbar^{\kappa-\chi(u)}\cdot\mathbf{y}_0,
\end{equation}
where $u$ ranges over curves of index 0, and $\natural(u)\in\mathbb{Z}$. 

Since $C$ is disjoint from $\star$, $\langle u,\star\rangle$ is constant for any 1-parameter family of $u$. Therefore, by analyzing the degeneration of index-1 moduli spaces, we see that $CW(\sqcup_{i}T_{q_i}^*\Sigma)_c$ is an $A_\infty$-algebra. Proposition \ref{lemma-algebra} can then be improved to:

\begin{proposition}
\label{lemma-algebra-c}
The $A_\infty$-algebra $CW(\sqcup_{i}T_{q_i}^*\Sigma)_c$ is supported in degree zero, and hence is an ordinary algebra.
\end{proposition}

\section{The evaluation map}
\label{section-f}

Following Abouzaid \cite{abouzaid2012wrapped}, we construct the evaluation map
\begin{equation*}
    \mathcal{F}\colon CW(\sqcup_{i}T_{q_i}^*\Sigma)_c \to \mathrm{H}_\kappa(\Sigma) \otimes_{\mathbb{Z}[\hbar]} \mathbb{Z}[[\hbar]]
\end{equation*}
in Section \ref{subsection-ev}.
It is given by counting holomorphic curves between cotangent fibers and the zero section of $T^*\Sigma$ in the framework of HDHF. 
We then show that $\mathcal{F}$ is a homomorphism of algebras in Section \ref{subsection-homo}. 
The key ingredient is the holomorphic curve interpretation of the HOMFLY skein relation due to Ekholm-Shende \cite{ekholm2021skeins}.
We finally prove Theorem \ref{thm-main} which states that the map $\mathcal{F}$ is an isomorphism.

\subsection{The definition}
\label{subsection-ev}
At this point we rename $D_m$ as $T_{m-1}$, where $\partial_i T_{m-1} = \partial_i D_m$ for $i=0,\dots, m$. 
The disk $T_{m-1}$ will be the $A_\infty$ base direction, where $p_1,\dots, p_{m-1}\in \partial D_m =\partial T_{m-1}$ correspond to inputs ($\kappa$-tuples of intersection points) and $\partial_m D_m =\partial_m T_{m-1}$ corresponds to the output ($\kappa$-tuples of arcs on $\Sigma$).  We view all of the strip-like ends $e_i$, $i=0,\dots, m$, of $T_{m-1}$ corresponding to $p_i$ as positive ends $[0,\infty)_{s_i}\times[0,1]_{t_i}$.  
Let $\mathcal{T}_{m-1}$ be the moduli space of $T_{m-1}$ modulo automorphisms; again we choose representatives $T_{m-1}$ of equivalence classes of $\mathcal{T}_{m-1}$ in a smooth manner.  


Let $\pi_{T^*\Sigma}$ be the projection $T_{m-1}\times {T^*\Sigma}\to {T^*\Sigma}$. Choose a sufficiently generic consistent collection $T_{m-1}\mapsto J_{T_{m-1}}$ of compatible almost complex structures on $T_{m-1}\times T^*\Sigma$ for all $T_{m-1}\in \mathcal{T}_{m-1}$ and all $m\geq 2$ such that:
\begin{enumerate}
    \item[(J1')] on each fiber $\pi_{T_{m-1}}^{-1}(p)=\{p\}\times T^*\Sigma$, $J_{T_{m-1}}$ restricts to an element of $\mathcal{J}_{T^*\Sigma,\alpha_{std}}$;
    \item[(J2')] $J_{T_{m-1}}$ projects holomorphically onto $T_{m-1}$;
    \item[(J3')] over each strip-like end $[0,\infty)_{s_i}\times[0,1]_{t_i}$, for $s_i$ sufficiently positive, $J_{T_{m-1}}$ is invariant in the $s_i$-direction and takes $\partial_{s_i}$ to $\partial_{t_i}$.
\end{enumerate}

Recall the time-$t$ Hamiltonian flow of $\phi^t_{H_V}$ (\ref{eq-H}). 
We will refer to $\Sigma$ as the zero section of $T^*\Sigma$ when it is clear from the context.
Let $\mathbf{q}$ (resp.\ $\mathbf{q}'$) be the set of intersection points between $\sqcup_{i}T_{q_i}^*\Sigma$ (resp.\ $\phi^1_{H_V}(\sqcup_{i}T_{q_i}^*\Sigma)$)  and $\Sigma$ and let $\mathbf{y}\in CF(\phi^1_{H_V}(\sqcup_{i}T_{q_i}^*\Sigma),\sqcup_{i}T_{q_i}^*\Sigma)$. Note that there is a canonical identification between ${\bf q}$ and ${\bf q}'$, 


We define $\mathcal{H}(\mathbf{q}',\mathbf{y},\mathbf{q})$ as the moduli space of maps
\begin{equation*}
    u\colon(\dot F,j)\to(T_1\times {T^*\Sigma},J_{T_1}),
\end{equation*}
where $(F,j)$ is a compact Riemann surface with boundary, $\mathbf{p}_0,\mathbf{p}_1,\mathbf{p}_2$ are disjoint $\kappa$-tuples of boundary marked points of $F$, $\dot F=F\setminus\cup_i \mathbf{p}_i$, and $u$ satisfies:
$$
\left\{
    \begin{tabular}{lll}
        $du\circ j=J_{T_1}\circ du$;\\
        
        $\pi_{T^*\Sigma}\circ u(z)\in\phi^1_{H_V}(\sqcup_{i}T_{q_i}^*\Sigma)$ if $\pi_{T_1}\circ u(z)\subset\partial_0 T_1$;\\ 
        each component of $\partial\dot F$ that projects to $\partial_0 T_1$ maps to a distinct $\phi^1_{H_V}(T^*_{q_i}\Sigma)$;\\
        $\pi_{T^*\Sigma}\circ u(z)\in\sqcup_{i}T_{q_i}^*\Sigma$ if $\pi_{T_1}\circ u(z)\subset\partial_1 T_1$;\\
        each component of $\partial\dot F$ that projects to $\partial_1 T_1$ maps to a distinct $T^*_{q_i}\Sigma$;\\
        $\pi_{T^*\Sigma}\circ u(z)\in\Sigma$ if $\pi_{T_1}\circ u(z)\subset\partial_2 T_1$;\\
        $\pi_{T^*\Sigma}\circ u$ tends to $\mathbf{q}'$, $\mathbf{y}$, $\mathbf{q}$  as $s_0, s_1,s_2\to +\infty$;\\
        $\pi_{T_1}\circ u$ is a $\kappa$-fold branched cover of a fixed $T_1\in\mathcal{T}_1$.
    \end{tabular}
\right.
$$
See Figure \ref{fig-F}.

\begin{figure}[ht]
    \centering
    \includegraphics[width=5cm]{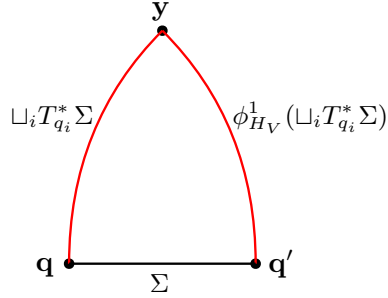}
    \caption{The $A_\infty$ base direction $T_1$ for $\mathcal{H}(\mathbf{q}',\mathbf{y},\mathbf{q})$. The notation denotes the corresponding preimages of $\pi_{T_1}$ in $T^*\Sigma$, e.g., $p_0$ is denoted by $\mathbf{q}'$ since $\pi_{T^*\Sigma}\circ u$ tends to $\mathbf{q}'$ as $\pi_{T_1}\circ u\to p_0$.}
    \label{fig-F}
\end{figure}

\begin{lemma}
    \label{lemma-dim-triangle}
    Fixing generic $J_{T_1}$, $\mathcal{H}(\mathbf{q}',\mathbf{y},\mathbf{q})$ is of dimension 0 and consists of discrete regular curves for all $\mathbf{q}$, $\mathbf{y}$ and $\mathbf{q}'$.
\end{lemma}
\begin{proof}
By the discussion before Lemma \ref{lemma-grading}, $|q_1|=\dots=|q_\kappa|=|q'_1|=\dots=|q'_\kappa|$=0. 
Hence $|\mathbf{q}|=|\mathbf{q}'|=0$. 
By Lemma \ref{lemma-grading}, $|\mathbf{y}|=0$. 
We then see that the virtual dimension of $\mathcal{H}(\mathbf{q}',\mathbf{y},\mathbf{q})$ is 0 by the index formula (\ref{eq-grading}); note that the same index formula holds even when we do not assume that the copies of the zero section $\Sigma$ are disjoint.
The lemma then follows from standard transversality arguments.
\end{proof}

Let $\mathcal{H}^\chi(\mathbf{q}',\mathbf{y},\mathbf{q})$ be the subset of $\mathcal{H}(\mathbf{q}',\mathbf{y},\mathbf{q})$ such that $\chi(\dot F)=\chi$. 

\begin{lemma}
    \label{lemma-H-T-1}
    Given $\mathbf{q}$, $\mathbf{y}$ and $\mathbf{q}'$, the moduli space $\mathcal{H}^\chi(\mathbf{q}',\mathbf{y},\mathbf{q})$ consists of finitely many curves for each Euler characteristic $\chi$.
\end{lemma}
\begin{proof} Each $\mathbf{y}$ determines a unique $\mathbf{q}$ and $\mathbf{q}'$. Since there is an energy bound for curves in $\mathcal{H}(\mathbf{q}',\mathbf{y},\mathbf{q})$, by Gromov compactness, $\mathcal{H}(\mathbf{q}',\mathbf{y},\mathbf{q})$ contains a finite number of curves for each $\chi$.
\end{proof}

Fix a parametrization of the arc $\partial_2 T_1$ from $p_0$ to $p_2$ by $\tau\colon [0,1] \to \partial_2 T_1$.
There exists a sufficiently generic consistent collection $\{T_{m}\mapsto J_{T_{m}}\}$ such that for all $u\in\mathcal{H}(\mathbf{q}',\mathbf{y},\mathbf{q})$, $(\pi_{T^*\Sigma}\circ u)\circ(\pi_{T_1}\circ u)^{-1}\circ\tau(t)$ consists of $\kappa$ distinct points on $\Sigma\setminus\{\star\}$ for each $t\in[0,1]$ and hence gives a path in $\mathrm{UConf}_{\kappa}(\Sigma\setminus\{\star\})$:
$$\gamma(u)\colon [0,1] \to \mathrm{UConf}_{\kappa}(\Sigma\setminus\{\star\}), \quad t\mapsto (\pi_{T^*\Sigma}\circ u)\circ(\pi_{T_1}\circ u)^{-1}\circ\tau(t).$$
This is possible since by the usual transversality argument the evaluation map at any point on the domain of $u$ is a submersion subject to the given Lagrangian boundary constraints.
Define 
\begin{equation*}
    \label{eq-confstarpath}
    \Omega(\mathrm{UConf}_{\kappa}(\Sigma\setminus\{\star\}),\mathbf{q}',\mathbf{q})=\{\gamma\colon [0,1] \to \mathrm{UConf}_{\kappa}(\Sigma \setminus \{\star\}) ~|~ \gamma(0)=\mathbf{q}',\gamma(1)=\mathbf{q}\}.
\end{equation*}
Then $\gamma(u) \in \Omega(\mathrm{UConf}_{\kappa}(\Sigma\setminus\{\star\}),\mathbf{q}',\mathbf{q})$.

For $u\in\mathcal{H}(\mathbf{q}',\mathbf{y},\mathbf{q})$, we can define $[\pi_{\Sigma}(u)]'$ as in Section \ref{section-c}: We extend the image $\pi_\Sigma(u)$ by the homotopies $B_i$ from Section \ref{section-c}, where $\partial_i\dot{F}$ ranges over all boundary arcs of $\partial \dot{F}$ which are not of ``output type''.
Here $\partial_i \dot{F}$ is of ``output type'' if $\pi_{T_1}\circ u(\partial_i \dot{F})\subset{\partial_2 T_1}$. 
We then define the algebraic intersection number
\begin{equation*}
    \langle u,\star\rangle\coloneqq\langle [\pi_\Sigma(u)]',\star\rangle.
\end{equation*}

We now define the evaluation map 
\begin{align}
\label{eq-E}
    \mathcal{E}\colon CW(\sqcup_{i}T_{q_i}^*\Sigma)_c&\to C_0(\Omega(\mathrm{UConf}_{\kappa}(\Sigma\setminus\{\star\}),\mathbf{q}',\mathbf{q}))\otimes \mathbb{Z}[c^{\pm1}]\otimes \mathbb{Z}[[\hbar]],\\
    \mathbf{y}&\mapsto\sum_{u\in\mathcal{H}(\mathbf{q}',\mathbf{y},\mathbf{q})} (-1)^{\natural(u)}\cdot c^{2\langle u,\star\rangle}\cdot\hbar^{\kappa-\chi(u)}\cdot \gamma(u).\nonumber
\end{align}




Since the perturbation term $V$ of (\ref{eq-H}) has small $W^{1,2}$-norm, the Hamiltonian vector field $X_{H_V}$ has small norm near the zero section $\Sigma$ and hence $\mathbf{q}'$ is close to $\mathbf{q}$. 
We then choose nonintersecting short paths $\gamma_i$ on $\Sigma$ from $q_i$ to $q'_i$ for $i=1,\dots,\kappa$. 
By pre-concatenating with $\{\gamma_i\}$, we can identify $\Omega(\mathrm{UConf}_{\kappa}(\Sigma\setminus\{\star\}),\mathbf{q}',\mathbf{q})$ with $\Omega(\mathrm{UConf}_{\kappa}(\Sigma\setminus\{\star\}),\mathbf{q})$ as in (\ref{eq-confstar}). 
By Definitions \ref{def-brc} and \ref{def-hecke} there is a natural projection map
\begin{align}
\label{eq-quotient}
    \mathcal{P}\colon C_0(\Omega(\mathrm{UConf}_{\kappa}(\Sigma\setminus\{\star\}),\mathbf{q}))\otimes \mathbb{Z}[c^{\pm1}]\otimes \mathbb{Z}[[\hbar]] &\to \mathrm{H}_\kappa(\Sigma,\mathbf{q})\otimes_{\mathbb{Z}[\hbar]} \mathbb{Z}[[\hbar]],
\end{align}
by first taking the $c$-deformed homotopy class and then quotienting by the skein relation. 
We finally define the evaluation map 
\begin{equation} \label{eq-F}
\mathcal{F}=\mathcal{P}\circ\mathcal{E}\colon CW(\sqcup_{i}T_{q_i}^*\Sigma)_c \to  \mathrm{H}_\kappa(\Sigma,\mathbf{q})\otimes_{\mathbb{Z}[\hbar]} \mathbb{Z}[[\hbar]].
\end{equation}
Note that $\mathcal{E}$ depends on the choice of the parametrization $\tau$ but $\mathcal{F}$ does not.

\subsection{The isomorphism}
\label{subsection-homo}
Both the domain and target of the map $\mathcal{F}$ in (\ref{eq-F}) are ordinary algebras. We will show that $\mathcal{F}$ is an isomorphism of algebras in this subsection. 

We consider a moduli space of curves whose $A_\infty$ base direction is $T_2\in \mathcal{T}_2$; see Figure \ref{fig-T2}. Let $\mathbf{q}''$ be the set of intersection points between $\phi^2_{H_V}(\sqcup_{i}T_{q_i}^*\Sigma)$ and $\Sigma$; there are canonical identifications ${\bf q}\simeq {\bf q}'\simeq {\bf q}''$.   
\begin{figure}[ht]
    \centering
    \includegraphics[width=6cm]{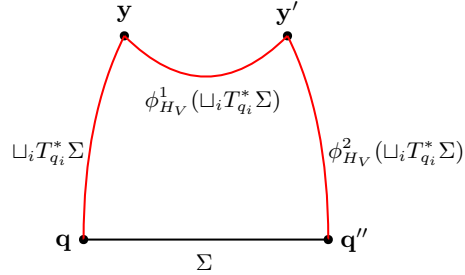}
    \caption{The $A_\infty$ base direction $T_2$ for $\mathcal{H}(\mathbf{q}'',\mathbf{y}',\mathbf{y},\mathbf{q})$.}
    \label{fig-T2}
\end{figure}
In this case, a 1-parameter family of $T_2\in\mathcal{T}_2$ may degenerate into broken curves in $\partial \mathcal{T}_2$ as shown in Figure \ref{fig-broken}.
As in Section \ref{subsection-wrap}, we need to construct the consistent collection carefully near the point $T_2'$ of $\partial\mathcal{T}_2$ corresponding to the left-hand side of Figure \ref{fig-broken}. 

\begin{figure}[ht]
    \centering
    \includegraphics[width=12cm]{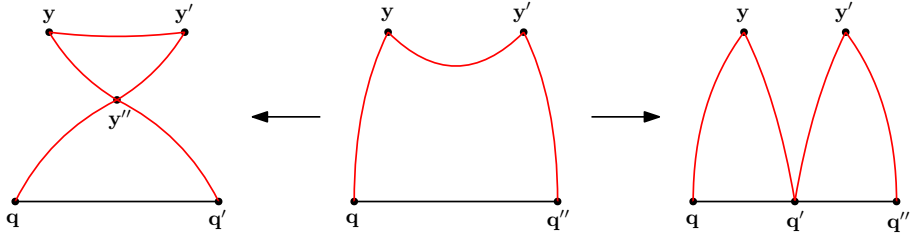}
    \caption{Degeneration of $\mathcal{H}^{\operatorname{ind}=1,\chi}(\mathbf{q}'',\mathbf{y}',\mathbf{y},\mathbf{q})$ in the $T_2$-direction: concatenation of curves in $\mathcal{M}^{\operatorname{ind}=0,\chi'}(\mathbf{y}',\mathbf{y},\mathbf{y}'')$ and $\mathcal{H}^{\operatorname{ind}=0,\chi''}(\mathbf{q}',\mathbf{y}'',\mathbf{q})$ (left); concatenation of curves in $\mathcal{H}^{\operatorname{ind}=0,\chi'}(\mathbf{q}'',\mathbf{y}',\mathbf{q}')$ and $\mathcal{H}^{\operatorname{ind}=0,\chi''}(\mathbf{q}',\mathbf{y},\mathbf{q})$ (right).}
    \label{fig-broken}
\end{figure}

We have already constructed consistent collections $J''_{D_m}$ and $\tilde L''_j(D_m)$ for $\mathcal{A}_m$.  Let us also assume that we have constructed $J_{T_m}$ for all $T_m$ on large compact subsets of $\mathcal{T}_m$ such that $J_{T_m}$ are consistent with $J''_{D_1}$ at the ends $e_i$, $i=1,\dots,m$. 
Denote by $\mathcal{U}$ the set of all $T_2\in \mathcal{T}_2$ that are sufficiently close to breaking into $D_2$ and $T_1$. Given $T_2\in \mathcal{U}$ we choose the neck region $N(T_2)$ (chosen to be smoothly dependent on $T_2$) and write $T_2^+$ and $T_2^-$ for the two components of $T_2\setminus N(T_2)$ that can be viewed as subsets of $D_2$ and $T_1$, respectively.  We then define $J_{T_2}$ to equal $J_{T_1}$ on $T_2^-\cup N(T_2)$ and $J''_{D_2}$ on $T_2^+\cup N(T_2)$, and note that $J''_{D_2}$ and $J_{T_1}$ agree on $N(T_2)$ by definition. 

We also need to choose a way of wrapping the Lagrangians corresponding to the boundaries of $T_2$ so that it is consistent with the breaking into $D_2\cup T_1$ as in the left-hand side of Figure~\ref{fig-broken}.  Note that as we degenerate to $D_2\cup T_1$, the boundary condition $\phi_{H_V}^2(\sqcup_{i}T_{q_i}^*\Sigma)$ approaches $\phi_{H_V}(\sqcup_{i}T_{q_i}^*\Sigma)$.
To keep track of this degeneration we define $\Psi_j\colon\partial_j T_2\to\mathrm{Symp}(T^*\Sigma,\omega)$, $j=0,1,2$, which depend smoothly on $T_2\in\mathcal{T}_2$ and such that:
\begin{align*}
    \Psi_j=\left\{
        \begin{array}{ll}
            \text{$\phi^{2-j}_{H_V}$ for $T_2\in\mathcal{T}_2\setminus \mathcal{U}$ and all $j$;}\\
            \text{$\psi^2\circ \phi^{2-j}_{H_V}$ on $\partial_j T_2\cap T_2^+$ for $T_2\in \mathcal{U}'$ and all $j$;}\\
            \text{$\phi^1_{H_V}$ on $\partial_0 T_2\cap T_2^-$ and $\operatorname{id}$ on $\partial_2 T_2\cap T_2^-$ for  $T_2\in \mathcal{U}'$;}
        \end{array}
    \right.
\end{align*}
where $\mathcal{U}'$ is a slight retraction of $\mathcal{U}$. 

Given generators 
$$\mathbf{y}'\in CF(\phi^2_{H_V}(\sqcup_{i}T_{q_i}^*\Sigma),\phi^1_{H_V}(\sqcup_{i}T_{q_i}^*\Sigma))_c,\quad \mathbf{y}\in CF(\phi^1_{H_V}(\sqcup_{i}T_{q_i}^*\Sigma),\sqcup_{i}T_{q_i}^*\Sigma)_c,$$ 
let $\mathcal{H}(\mathbf{q}'',\mathbf{y}',\mathbf{y},\mathbf{q})$ be the moduli space of maps $u\colon(\dot F,j)\to(T_2\times {T^*\Sigma},J_{T_2})$, where $(F,j)$ is a compact Riemann surface with boundary, $T_2\in\mathcal{T}_2$, $\mathbf{p}_0,\dots,\mathbf{p}_3$ are disjoint $\kappa$-tuples of boundary marked points of $F$ and $\dot F=F\setminus\cup_i \mathbf{p}_i$ so that $u$ satisfies:
\begin{align*}
    \left\{
        \begin{array}{ll}
            \text{$du\circ j=J_{T_2}\circ du$;}\\
            \text{$\pi_{T^*\Sigma}\circ u(z)\in \Psi_j(\pi_{T_2}\circ u(z))(\sqcup_{i}T_{q_i}^*\Sigma)$ if $\pi_{T_2}\circ u(z)\subset\partial_j T_2$ for $j=0,1,2$;}\\
            \text{each component of $\partial\dot F$ that projects to $\partial_j T_2$ maps to a distinct}\\
            \qquad \text{$\Psi_j(\pi_{T_2}\circ u(z))(T^*_{q_i}\Sigma)$ for $j=0,1,2$;}\\
            \text{$\pi_{T^*\Sigma}\circ u(z)\in\Sigma$ if $\pi_{T_2}\circ u(z)\subset\partial_3 T_2$;}\\
            \text{$\pi_{T^*\Sigma}\circ u$ tends to $\mathbf{y}'$, $\mathbf{y}$, $\mathbf{q}$  as $s_1,s_2,s_3\to+\infty$;}\\
			\text{$\pi_{T^*\Sigma}\circ u$ tends to a $\kappa$-tuple ${\bf q}''(u)$ of intersection points of $\Psi_0(\pi_{T_2}\circ u(z))(\sqcup_{i}T_{q_i}^*\Sigma)$,}\\
			\qquad \text{for $\pi_{T_2}\circ u(z)\subset\partial_0 T_2$, and $\Sigma$, in bijection with $\mathbf{q}''$ as $s_0\to +\infty$;}\\
            \text{$\pi_{T_2}\circ u$ is a $\kappa$-fold branched cover of some $T_2\in\mathcal{T}_2$.}
        \end{array}
    \right.
\end{align*}

\begin{remark} \label{rmk: identifications}
All the $\kappa$-tuples ${\bf q}''(u)$ are close to ${\bf q}''$ since $V(t,q)$ is small, and hence there is a family of diffeomorphisms that takes ${\bf q}''(u)$ to ${\bf q}''$ and is smoothly varying with ${\bf q}''(u)$. 
\end{remark}
\begin{lemma}
    \label{lemma-H-T-2}
    There exists a sufficiently generic consistent collection of almost complex structures such that $\mathcal{H}(\mathbf{q}'',\mathbf{y}',\mathbf{y},\mathbf{q})$ is of dimension 1 and is transversely cut out for all $\mathbf{q}$, $\mathbf{y}$, $\mathbf{y}'$ and $\mathbf{q}''$. Moreover, $\mathcal{H}(\mathbf{q}'',\mathbf{y}',\mathbf{y},\mathbf{q})$ admits a compactification $\overline{\mathcal{H}}(\mathbf{q}'',\mathbf{y}',\mathbf{y},\mathbf{q})$ such that its boundary $\partial\overline{\mathcal{H}}(\mathbf{q}'',\mathbf{y}',\mathbf{y},\mathbf{q})$ is of dimension 0 and contains discrete broken or nodal curves.
\end{lemma}

\begin{proof} Similar to Lemma \ref{lemma-dim-triangle}, since $|\mathbf{y}|=|\mathbf{y}'|=0$ and $|\mathbf{q}|=|\mathbf{q}''|$, $\mathcal{H}(\mathbf{q}'',\mathbf{y}',\mathbf{y},\mathbf{q})$ has virtual dimension 0 by (\ref{eq-grading}). 
By standard transversality arguments and Gromov compactness, a 1-parameter family of curves in $\mathcal{H}(\mathbf{q}'',\mathbf{y}',\mathbf{y},\mathbf{q})$ may limit to broken curves by pinching boundaries of $T_2$ or to nodal curves by letting branch points approach $\partial T_2$. 
\end{proof}

We now define a map from $\mathcal{H}(\mathbf{q}'',\mathbf{y}',\mathbf{y},\mathbf{q})\times[0,1]$ to the set of subsets of $\Sigma$ by
\begin{equation}
\label{eq-path}
    \gamma(u)(t)=(\pi_{T^*\Sigma}\circ u)\circ(\pi_{T_2}\circ u)^{-1}\circ\tau(t),
\end{equation}
where $\tau\colon [0,1] \to \partial_3 T_2$ parametrizes the boundary arc $\partial_3 T_2$ from $p_0$ to $p_3$.   Here we are identifying all the ${\bf q}''(u)$ with ${\bf q}''$ using Remark~\ref{rmk: identifications}.

Let
\begin{align*}
    \mathcal{H}_0(\mathbf{q}'',\mathbf{y}',\mathbf{y},\mathbf{q})&=\{u\in\mathcal{H}(\mathbf{q}'',\mathbf{y}',\mathbf{y},\mathbf{q})\,|\,\gamma(u)(t)\text{ contains} \\
    &\text{$\kappa$ distinct points of $\Sigma$ and $\gamma(u)(t)\cap\{\star\}=\varnothing$ for all $t$}\,\}.
\end{align*}
In a manner similar to the definition of $\mathcal{F}$ in (\ref{eq-F}), we define the evaluation map
\begin{gather*}
    \mathcal{G}\colon\mathcal{H}_0(\mathbf{q}'',\mathbf{y}',\mathbf{y},\mathbf{q})\to\mathrm{H}_\kappa(\Sigma,\mathbf{q})\otimes_{\mathbb{Z}[\hbar]} \mathbb{Z}[[\hbar]],\\
    u\mapsto (-1)^{\natural(u)}\cdot c^{2\langle u,\star\rangle}\cdot\hbar^{\kappa-\chi(u)}\cdot [\gamma(u)].
\end{gather*}
Here we are writing $[\gamma]=\mathcal{P}(\gamma)$, where $\mathcal{P}(\gamma)$ is defined by (\ref{eq-quotient}).  As before, $\langle u,\star\rangle\coloneqq\langle [\pi_\Sigma(u)]',\star\rangle$, where ${[{\pi}_\Sigma(u)]'}$ is as in Section \ref{subsection-ev} with one difference: we replace all the $\alpha_i= \pi_\Sigma\circ u(\partial_i \dot F)$, {\em except those that have boundary condition $u(\partial_i\dot F)\subset \Sigma$,} by their piecewise $V$-geodesic replacements. The algebraic intersection number $\langle u,\star\rangle$ is well-defined since the boundary of ${[{\pi}_\Sigma(u)]'}$ is disjoint from $\star$ when $u\in\mathcal{H}_0(\mathbf{q}'',\mathbf{y}',\mathbf{y},\mathbf{q})$. 

\begin{proposition}
    \label{prop-algebra}
    The map $\mathcal{F}$ in (\ref{eq-F}) is a homomorphism of algebras.
\end{proposition}
\begin{proof} 
It suffices to show that
\begin{equation*}
    \mathcal{F}(\mu^2(\mathbf{y},\mathbf{y}'))=\mathcal{F}(\mathbf{y})\mathcal{F}(\mathbf{y}')
\end{equation*}
for any $\mathbf{y},\mathbf{y}'\in CW(\sqcup_{i}T_{q_i}^*\Sigma)_c$. We analyze the boundary of the index 1 moduli space $\overline{\mathcal{H}}^\chi(\mathbf{q}'',\mathbf{y}',\mathbf{y},\mathbf{q})$ for varying $\chi$.

For each $u\in{\mathcal{H}}^\chi(\mathbf{q}'',\mathbf{y}',\mathbf{y},\mathbf{q})$, consider the path $\gamma(u)$ defined by (\ref{eq-path}) as a $\kappa$-tuple of paths in $\Sigma$.
A single generic $u$ belongs to ${\mathcal{H}_0}^\chi(\mathbf{q}'',\mathbf{y}',\mathbf{y},\mathbf{q})$ and $\mathcal{G}(u)$ is well-defined but for a generic $1$-parameter family $u_t\in {\mathcal{H}}^\chi(\mathbf{q}'',\mathbf{y}',\mathbf{y},\mathbf{q})$, $t\in[0,1]$, $\gamma(u_t)$ may intersect $\star$ at some $t\in(0,1)$. 
At any rate $\mathcal{G}(u_0)=\mathcal{G}(u_1)$ since we are taking values in $\mathrm{H}_\kappa(\Sigma,\mathbf{q})$ and the marked point relation compensates for each extra intersection with $\star$. Hence we may ignore the intersections with $\star$ in this proof.


Note that codimension-1 degenerations only occur in the $A_\infty$ base direction, i.e., the projection of curves to the $A_\infty$ base direction leads to codimension-1 degenerations of $D_m$. The schematic picture is shown in Figures \ref{fig-broken} and \ref{fig-nodal}.

There are three types of boundary degenerations:
\begin{enumerate}
    \item[(1)]  $\coprod_{\mathbf{y}'',\chi'+\chi''-\kappa=\chi} \mathcal{M}^{\operatorname{ind}=0,\chi'}(\mathbf{y}',\mathbf{y},\mathbf{y}'') \times\mathcal{H}^{\operatorname{ind}=0,\chi''}(\mathbf{q}',\mathbf{y}'',\mathbf{q})$;
    \item[(2)] $\coprod_{\mathbf{q}'',\chi'+\chi''-\kappa=\chi}\mathcal{H}^{\operatorname{ind}=0,\chi'}(\mathbf{q}'',\mathbf{y}',\mathbf{q}')\times\mathcal{H}^{\operatorname{ind}=0,\chi''}(\mathbf{q}',\mathbf{y},\mathbf{q})$;
    \item[(3)] the set $\partial_n \overline{\mathcal{H}}^{\operatorname{ind}=1, \chi}(\mathbf{q}'',\mathbf{y}',\mathbf{y},\mathbf{q})$ with a nodal degeneration along $\Sigma$.
\end{enumerate}
(1) is given on the left-hand side of Figure \ref{fig-broken} and contributes $\mathcal{F}(\mu^2(\mathbf{y},\mathbf{y}'))$. (2) is given on the right-hand side of Figure \ref{fig-broken} and contributes $\mathcal{F}(\mathbf{y})\mathcal{F}(\mathbf{y}')$.
A standard gluing argument shows that all contributions to $\mathcal{F}(\mu^2(\mathbf{y},\mathbf{y}'))$ and $\mathcal{F}(\mathbf{y})\mathcal{F}(\mathbf{y}')$ come from such broken degenerations. 

\begin{figure}[ht]
    \centering
    \includegraphics[width=12cm]{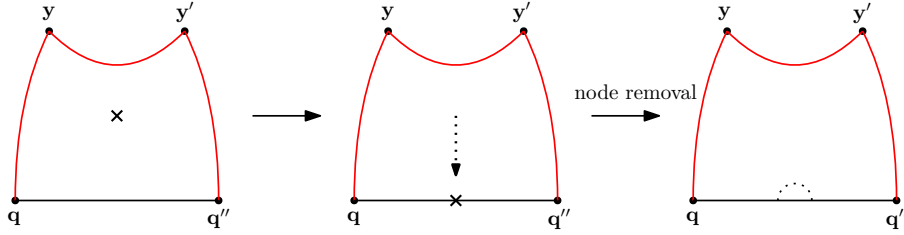}
    \caption{Nodal degeneration of $\mathcal{H}^{\operatorname{ind}=1,\chi}(\mathbf{q}'',\mathbf{y}',\mathbf{y},\mathbf{q})$ in the $T_2$ direction: a nodal point on $\Sigma$ (middle); removal of the nodal point (right).}
    \label{fig-nodal}
\end{figure}

We now discuss (3), which is given in Figure \ref{fig-nodal}. 
Let $u_s$, $s\in[0,1)$, be a generic 1-parameter family such that $\pi_{T_2}\circ u_s$ has a branch point (generically a double branch point) that limits to $\partial T_2$ as $s\to 1$.
We denote this limiting curve as $u_n\in\partial_{n}\overline{\mathcal{H}}^{\operatorname{ind}=1,\chi}(\mathbf{q}'',\mathbf{y}',\mathbf{y},\mathbf{q})$.
Moreover, the only component of $\partial T_2$ that a branch point can approach is $\partial_3 T_2$ (corresponding to the zero section $\Sigma$) since all other boundary arcs correspond to disjoint sets of Lagrangians where nodal degenerations cannot occur. 
By Gromov compactness, $\partial_{n}\overline{\mathcal{H}}^{\operatorname{ind}=1,\chi}(\mathbf{q}'',\mathbf{y}',\mathbf{y},\mathbf{q})$ is finite.
If we continue this family past the nodal curve $u_n$, then the nodal point is removed and $\chi$ increases by 1 as seen on the right-hand side of Figure \ref{fig-nodal}. 
Interpreted in another way, given $u_n\in \partial_n\overline{\mathcal{H}}^{\operatorname{ind}=1,\chi}(\mathbf{q}'',\mathbf{y}',\mathbf{y},\mathbf{q})$, there exists a 1-parameter family of curves $u'_t$, $t\in(-\delta,\delta)$, in $\mathcal{H}^{\operatorname{ind}=1,\chi+1}(\mathbf{q}'',\mathbf{y}',\mathbf{y},\mathbf{q})$ such that $u'_0$ is obtained from $u_n$ by removing the nodal point and the family $\gamma(u'_t)$ corresponds to a crossing of two paths in $[0,1]\times \Sigma$.
This is illustrated in Figure \ref{fig-skein-brane} as a skein relation on $[0,1]\times\Sigma$.
\begin{figure}[ht]
    \centering
    \includegraphics[width=10cm]{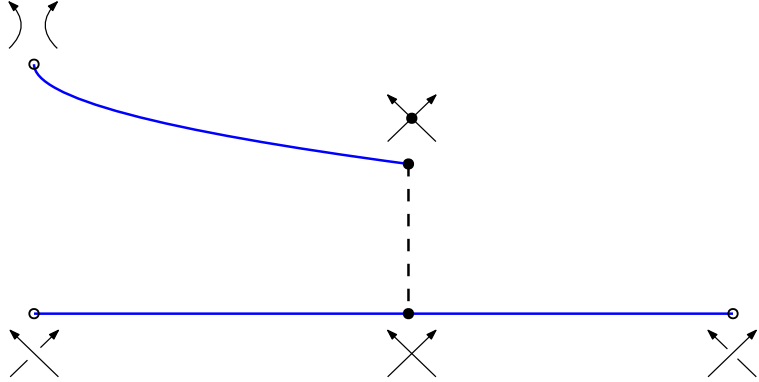}
    \caption{Representing a curve $u$ by its image $\gamma(u)$ on $[0,1]\times\Sigma$. The upper arc denotes the moduli space with a nodal degeneration. The lower arc denotes the companion moduli space with $\chi$ increased by 1, whose evaluation by $\mathcal{G}$ on $[0,1]\times\Sigma$ exhibits a crossing of two braid strands.}
    \label{fig-skein-brane}
\end{figure}

On the other hand, given a 1-parameter family $u'_t$, $t\in(-\delta,\delta)$, of curves in $\mathcal{H}^{\operatorname{ind}=1,\chi+1}(\mathbf{q}'',\mathbf{y}',\mathbf{y},\mathbf{q})$ such that $\gamma(u'_t)$ exhibits a single crossing of two strands (see the bottom blue line of Figure \ref{fig-skein-brane}), let $u_n$ be the nodal curve corresponding to $u'_0$ as in the middle of Figure \ref{fig-nodal}. 
Suppose the nodal point on the domain of $u_n$ is $p_n\in\dot{F}$. 
Since a neighborhood of $u_n(p_n)$ in the ambient symplectic manifold $T_2\times T^*\Sigma$ is diffeomorphic to $T^*\mathbb{R}\times T^*\mathbb{R}^2\approx \mathbb{R}^6$, we can construct a preglued curve $\tilde{u}_{1-\epsilon}$ using a cut-off version of the standard hyperbolic node model of \cite[Section 4.1.1]{ekholm2021skeins}, where $\epsilon\in[0,\epsilon_0)$ is the pregluing parameter for some small $\epsilon_0>0$ so that $u_n=\tilde{u}_1$. 
By a standard Newton iteration technique, there exists a unique 1-parameter family of holomorphic curves $u_{1-\epsilon}\in\overline{\mathcal{H}}^{\operatorname{ind}=1,\chi}(\mathbf{q}'',\mathbf{y}',\mathbf{y},\mathbf{q})$ for each $\epsilon\in[0,\epsilon_0)$, where $u_n=u_1$; see \cite[Lemma 4.16]{ekholm2021skeins} for details. 
Therefore, a small neighborhood of $u_n$ in $\overline{\mathcal{H}}^{\operatorname{ind}=1,\chi}(\mathbf{q}'',\mathbf{y}',\mathbf{y},\mathbf{q})$ is homeomorphic to $[0,\epsilon_0)$, which corresponds to the upper arc of Figure \ref{fig-skein-brane}. 


Summarizing the above discussion, $\sqcup_\chi \overline{\mathcal{H}}^{\operatorname{ind}=1,\chi}(\mathbf{q}'',\mathbf{y}',\mathbf{y},\mathbf{q})$ is a uni-trivalent graph with trivalent vertices indexed by $u_n\in\partial_{n}\mathcal{H}^{\operatorname{ind}=1,\chi}(\mathbf{q}'',\mathbf{y}',\mathbf{y},\mathbf{q})$. 
We denote the trivalent neighborhood of $u_n$ (the two blue arcs in Figure \ref{fig-skein-brane}) by $\mathcal{N}_{u_n}$, where $\partial\overline{\mathcal{N}}_{u_n}$ contains three curves $u_+$, $u_-$ and $u_0$ that correspond to $\includegraphics[width=0.5cm,valign=c]{poscross.eps}$, $\includegraphics[width=0.5cm,valign=c]{negcross.eps}$ and $\includegraphics[width=0.5cm,valign=c]{node.eps}$ in Figure \ref{fig-skein-brane}, respectively.
Now 
\begin{equation*}
    \overline{\mathcal{H}}^{\operatorname{ind}=1}(\mathbf{q}'',\mathbf{y}',\mathbf{y},\mathbf{q})\backslash\bigsqcup_{u_n\in\partial_n{\mathcal{H}}^{\operatorname{ind}=1}(\mathbf{q}'',\mathbf{y}',\mathbf{y},\mathbf{q})} \mathcal{N}_{u_n}
\end{equation*}
is a 1-dimensional manifold with boundary. 
By comparing the count of its boundary curves with different Euler characteristics, we have
\begin{align}
    &\mathcal{F}(\mathbf{y})\mathcal{F}(\mathbf{y}')-\mathcal{F}(\mu^2(\mathbf{y},\mathbf{y}'))\label{eq-algebra}\\
    &=1\cdot[\sum_{\chi(u_{+})=\kappa}\mathcal{G}(u_+)-\sum_{\chi(u_-)=\kappa}\mathcal{G}(u_-)]\nonumber\\
    &+\hbar\cdot[\sum_{\chi(u_{+})=\kappa-1}\mathcal{G}(u_+)-\sum_{\chi(u_-)=\kappa-1}\mathcal{G}(u_-)-\sum_{\chi(u_{0})=\kappa-1}\mathcal{G}(u_0)]\nonumber\\
    &+\hbar^2\cdot[\sum_{\chi(u_{+})=\kappa-2}\mathcal{G}(u_+)-\sum_{\chi(u_-)=\kappa-2}\mathcal{G}(u_-)-\sum_{\chi(u_{0})=\kappa-2}\mathcal{G}(u_0)]\nonumber\\
    &+\cdots\nonumber
\end{align}
The skein relation in Definition \ref{def-hecke} implies that 
\begin{equation*}
    \mathcal{G}(u_+)-\mathcal{G}(u_-)=\hbar\cdot \mathcal{G}(u_0) \in \mathrm{H}_\kappa(\Sigma,\mathbf{q})\otimes_{\mathbb{Z}[\hbar]} \mathbb{Z}[[\hbar]].
\end{equation*}
So the right-hand side of (\ref{eq-algebra}) is zero.
This completes the proof.
\end{proof}

To show $\mathcal{F}$ is an isomorphism, it suffices to show that $\mathcal{F}$ is a bijection. We use a perturbation argument and start with the case where $\hbar=0$:
\begin{lemma}
    \label{lemma-iso-h-0}
    The restriction of $\mathcal{F}$ to $\hbar=0$ is an isomorphism:
    \begin{equation*}
        \mathcal{F}|_{\hbar=0}\colon HW(\sqcup_{i}T_{q_i}^*\Sigma)_c|_{\hbar=0}\to \mathrm{H}_\kappa(\Sigma,\mathbf{q})|_{\hbar=0}.
    \end{equation*}
\end{lemma}
\begin{proof}
For simplicity, we write $\mathcal{F}_0=\mathcal{F}|_{\hbar=0}$ throughout this proof.

We first prove the lemma for $\kappa=1$. 
In this case $\mathrm{H}_1(\Sigma,\mathbf{q})|_{\hbar=0}\cong\mathbb{Z}[\pi_1(\Sigma,q)_c]$.
Setting $c=1$, $\pi_1(\Sigma,q)_c$ is isomorphic to $\pi_1(\Sigma,q)$. 
In each homotopy class of $\pi_1(\Sigma,q)$, there is a unique generator ${y}\in HW(T_{q}^*\Sigma)$ whose Legendre transform $\mathcal{L}(y)$ represents this class. By \cite[Lemma 5.1]{abouzaid2012wrapped}, $\mathcal{F}_0|_{c=1}({y})=[\mathcal{L}(y)]$. Hence, $\mathcal{F}_0|_{c=1}$ is an isomorphism. 
Add the parameter $c$ and view ${y}\in HW(T_{q}^*\Sigma)_c$. We have $\mathcal{F}_0({y})=c^d[\mathcal{L}(y)]_c$ for some integer $d$, where $[\mathcal{L}(y)]_c \in \pi_1(\Sigma,q)_c$ is a lift of $[\mathcal{L}(y)] \in \pi_1(\Sigma,q)$. 
Hence $\mathcal{F}_0$ is an isomorphism. 
For later use, we denote the map $\mathcal{F}_0$ for $\kappa=1$ as
\begin{equation}
\label{eq-tildeFc}
    \tilde{\mathcal{F}}_c\colon HW(T_{q}^*\Sigma)_c \to \mathbb{Z}[\pi_1(\Sigma,q)_c],
\end{equation}
which is the version of (\ref{eq-F-Ab}) with the $c$-parameter.

We now prove the lemma for $\kappa\ge1$.
The $\mu^2$-operation of $HW(\sqcup_{i}T_{q_i}^*\Sigma)_c|_{\hbar=0}$ only counts curves with $\chi=\kappa$, i.e., where there are $\kappa$ trivial pseudoholomorphic disks. 
In this case one can easily compute that $HW(\sqcup_{i}T_{q_i}^*\Sigma)_c|_{\hbar=0}$ is isomorphic to $\left(\otimes_i HW(T_{q_i}^*\Sigma)_c\right)\rtimes S_{\kappa}$.
On the other hand, $\mathrm{H}_\kappa(\Sigma,\mathbf{q})_c|_{\hbar=0}$ degenerates to $\left(\otimes_i\,\mathbb{Z}[\pi_1(\Sigma,q_i)_c]\right) \rtimes S_{\kappa}$ by Lemma \ref{lemma-dim-Heckeh=0}. 
Here, both tensor products are over $\mathbb{Z}[c^{\pm1}]$.

Since $HW(\sqcup_{i}T_{q_i}^*\Sigma)_c|_{\hbar=0}$ is generated by $\otimes_i HW(T_{q_i}^*\Sigma)_c$ and $S_{\kappa}$ as an algebra, it suffices to show that 
\begin{enumerate} 
    \item $\mathcal{F}_0|_{\otimes_i HW(T_{q_i}^*\Sigma)_c}=\tilde{\mathcal{F}}_c^{\otimes \kappa}\colon \otimes_i HW(T_{q_i}^*\Sigma)_c \to \otimes_i\, \mathbb{Z}[\pi_1(\Sigma,q_i)_c]$, where the map $\tilde{\mathcal{F}}_c$ is in (\ref{eq-tildeFc});
    \item $\mathcal{F}_0|_{S_{\kappa}}=\mathrm{id}\colon S_{\kappa} \to S_{\kappa}$. 
\end{enumerate}
By Theorem \ref{thm-ab}, Remark \ref{rmk-ab} and \cite[Lemma 5.1]{abouzaid2012wrapped}, $\tilde{\mathcal{F}}_c$ maps the time-1 Hamiltonian flow to the homotopy class of its Legendre transform. 
The first equation follows since $\mathcal{F}_0|_{\otimes_i HW(T_{q_i}^*\Sigma)_c}$ maps $\kappa$ time-1 Hamiltonian flows to their Legendre transforms. 
    
The symmetric group $S_{\kappa}$ is generated by transpositions $\sigma_i=(i,i+1)$. 
Let $\mathbf{y}_i=\{y_{i1},\dots,y_{i\kappa}\} \in HW(\sqcup_{i}T_{q_i}^*\Sigma)_c|_{\hbar=0}$ be the corresponding generator in Floer homology, where $y_{ij} \in CF(\phi^1_{H_V}(T_{q_j}^*\Sigma), T_{q_{\sigma_i(j)}}^*\Sigma).$
Since $\hbar=0$, the map $\mathcal{F}_0(\mathbf{y}_i)$ counts curves with $\chi=\kappa$, that is, $\kappa$ holomorphic disks, where each disk limits to $y_{ij}$ at the positive puncture and evaluates on $\Sigma$ as a path from $q_j$ to $q_{\sigma_i(j)}$, for $j=1,\dots,\kappa$.

Note that each homotopy class of paths from $q_j$ to $q_{\sigma_i(j)}$ contains a unique $V$-perturbed geodesic, i.e., $\mathcal{L}(y_{ij})$ (see Definition \ref{def-leg-L}).
By \cite{abouzaid2012wrapped}, there is a unique holomorphic disk which limits to $y_{ij}$ at the positive puncture and evaluates on $\Sigma$ as a path $\gamma_j$ from $q_j$ to $q_{\sigma_i(j)}$ on the stable manifold of $\mathcal{L}(y_{ij})$.
However, since the action of $y_{ij}$ equals that of $\mathcal{L}(y_{ij})$, we see $\gamma_j=\mathcal{L}(y_{ij})$.

Therefore, the evaluation of the boundary of the unique $\kappa$-tuple of disks on the zero section $\Sigma\subset T^*\Sigma$ gives a based loop in $\mathrm{UConf}_{\kappa}(\Sigma,\mathbf{q})$, which consists of two short paths $\mathcal{L}(y_{i,i+1})$ and $\mathcal{L}(y_{i+1,i})$, and trivial paths from $q_j$ to $q_j$ for $j\neq i,i+1$. 
This loop gives a class in $\mathrm{H}_\kappa(\Sigma,\mathbf{q})_c|_{\hbar=0} \cong \left(\otimes_i\,\mathbb{Z}[\pi_1(\Sigma,q_i)_c]\right) \rtimes S_{\kappa}$ corresponding to $\sigma_i \in S_{\kappa}$.  
Hence, the second equation follows. 
\end{proof}

\vskip.1in
\begin{proof}[Proof of Theorem \ref{thm-main}]
$ $\vskip.1in\noindent
{\em Injectivity of $\mathcal{F}$.}
Suppose that there exists $\mathbf{a}\neq0$ such that $\mathcal{F}(\mathbf{a})=0$.
Let $\mathbf{a}=\sum_{i\geq0}\hbar^i\mathbf{a}_i$, where $\mathbf{a}_i \in CW(\sqcup_{i}T_{q_i}^*\Sigma)_c|_{\hbar=0}$. 
Since $\mathrm{H}_\kappa(\Sigma,\mathbf{q})$ has no $\hbar$-torsion, we may assume that $\mathbf{a}_0\neq0$. Setting $\hbar=0$, we have $\mathcal{F}(\mathbf{a}_0)=\mathcal{F}(\mathbf{a})=0$. 
Hence $\mathcal{F}|_{\hbar=0}(\mathbf{a}_0)=0$, which means $\mathbf{a}_0=0$ since $\mathcal{F}|_{\hbar=0}$ is an isomorphism. 
This leads to contradiction. 
Thus $\mathcal{F}$ is injective.

\vskip.1in\noindent
{\em Surjectivity of $\mathcal{F}$.} 
It suffices to show that any $\mathbf{b} \in \mathrm{H}_\kappa(\Sigma,\mathbf{q})$ is in $\mathrm{Im}\,\mathcal{F}$. Since $\mathcal{F}|_{\hbar=0}$ is an isomorphism, there exists $\mathbf{a}_0\in CW(\sqcup_{i}T_{q_i}^*\Sigma)_c|_{\hbar=0}$ such that $\mathcal{F}(\mathbf{a}_0)\equiv\mathbf{b}\,(\mathrm{mod\,\,\hbar})$. Let
\begin{equation*}
    \left.\mathbf{b}_1=\frac{(\mathbf{b}-\mathcal{F}(\mathbf{a}_0))}{\hbar}\right\vert_{\hbar=0}.
\end{equation*}
Then there exists $\mathbf{a}_1\in CW(\sqcup_{i}T_{q_i}^*\Sigma)_c|_{\hbar=0}$ such that $\mathcal{F}(\mathbf{a}_1)\equiv\mathbf{b}_1\,(\mathrm{mod\,\,\hbar})$. Repeating this procedure, 
we get $\mathcal{F}(\sum_{i\geq0}\mathbf{a}_i\hbar^i)=\mathbf{b}$. Hence $\mathcal{F}$ is surjective.
\end{proof}

\section{Surfaces with punctures}
\label{section-boundary}
In this section we show that Theorem \ref{thm-main} still holds for $\mathring{\Sigma}$, which is obtained from a closed oriented surface of genus $g\geq 0$ by removing a finite number ($>0$) of punctures. For simplicity we assume that $c=1$. 

In this case, the wrapped Floer homology of the cotangent fibers and the $\mathcal{F}$ map of (\ref{eq-F}) can be defined similarly but need modifications near the punctures. The main issue is the noncompactness of the moduli space of holomorphic curves: if the wrapped Lagrangians on $T^*\mathring{\Sigma}$ approach the punctures when projected to $\mathring{\Sigma}$, then a sequence of curves bounded by those wrapped Lagrangians projected to $\mathring{\Sigma}$ may also approach the punctures. 

To remedy this, we confine the wrapped Lagrangians and all involved holomorphic curves to stay over a compact subset of $\mathring{\Sigma}$. 
Our approach is a simple application of the partially wrapped Fukaya category by Sylvan \cite{sylvan2019partially} and its further development by Ganatra, Pardon and Shende \cite{ganatra2018sectorial,ganatra2020covariantly}. 
A similar approach in the context of sutured contact manifolds is due to Colin, Ghiggini, Honda and Hutchings \cite{colin2011sutures}.

Let $g$ be the standard flat metric on $\mathring{\Sigma}$ if $\mathring{\Sigma}$ is homeomorphic to $\mathbb{R}^2$ or $\mathbb{R}\times S^1$. 
Otherwise let $g$ be a complete finite-volume hyperbolic metric on $\mathring{\Sigma}$, i.e., all the punctures of $\mathring{\Sigma}$ correspond to cusps.

Suppose $\Sigma$ is a closed oriented surface and $z\in\Sigma$. We consider the once punctured surface $\mathring{\Sigma}=\Sigma\setminus\{z\}$ (the case of more than one puncture is similar). 

Let $\mathbf{q}=\{q_1,\dots,q_\kappa\}\subset{\mathring{\Sigma}}$ be a $\kappa$-tuple of points. 
Pick an end $\mathcal{N}\approx (-\infty,1)_s\times S^1_\theta$ near the puncture of $\mathring{\Sigma}$ so that $\{q_1,\dots,q_\kappa\}\subset{\mathring{\Sigma}}\setminus \mathcal{N}$. 
We have a trivialization of $T\mathcal{N}$ by $\{\frac{\partial}{\partial s},\frac{\partial}{\partial \theta}\}$. See Figure \ref{fig-cusp-end}. 
\begin{figure}[ht]
    \centering
    \includegraphics[width=9cm]{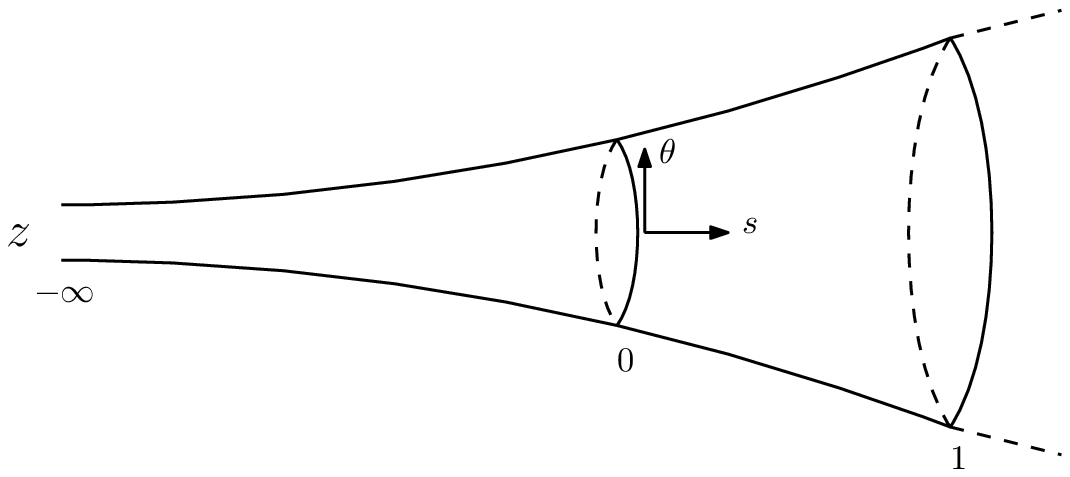}
    \caption{}
    \label{fig-cusp-end}
\end{figure}
Let $(p_s,p_\theta)$ be the dual coordinates to $(s,\theta)$ so that we have a trivialization of $T^*\mathcal{N}$ by $\{\frac{\partial}{\partial p_s},\frac{\partial}{\partial p_\theta}\}$. 
The canonical symplectic form $\omega$ on $\mathring{\Sigma}$ restricts to $\omega=ds\wedge dp_s+d\theta\wedge dp_\theta$ on $T^*\mathcal{N}$. 
We fix the trivial almost complex structure $J_\mathcal{N}$ on $T^*\mathcal{N}$ so that $J_\mathcal{N}(\frac{\partial}{\partial s})=\frac{\partial}{\partial p_s}$ and $J_\mathcal{N}(\frac{\partial}{\partial \theta})=\frac{\partial}{\partial p_{\theta}}$. 
Then there is a $(J_\mathcal{N},j)$-holomorphic map
\begin{equation*}
    \pi_s\colon T^*\mathcal{N}\to\mathbb{C}_{\mathrm{Re}<1},\quad (s,\theta,p_s,p_\theta)\mapsto s+j p_s,
\end{equation*}
where $j$ is the standard complex structure on $\mathbb{C}_{\mathrm{Re}<1}$.

Fix a diffeomorphism $f\colon(0,1)_s\to(-\infty,1)_s$ such that $f'>0$, $f''\leq 0$, and $f(s)=s$ for $s \geq 1/2$.  
It induces the diffeomorphism $\mathcal{N}\cap\{s>0\}\stackrel\sim\to \mathcal{N}$, $(s,\theta)\mapsto (f(s),\theta)$.
Extending by the identity, we obtain a diffeomorphism
\begin{equation*}
    \tilde{f}\colon\mathring{\Sigma}_0\coloneqq\mathring{\Sigma}\setminus \{s\leq0\}\to \mathring{\Sigma}.
\end{equation*}
The pullback metric $\tilde{f}^*g$ on $\mathring{\Sigma}_0$ induces a norm $|\cdot|_f$ on $T^*{\mathring{\Sigma}_0}$. Choose a time-dependent Hamiltonian $H_{V,f}\colon[0,1]\times T^*{\mathring{\Sigma}_0}\to\mathbb{R}$ with $|p|_f$ instead of $|p|$ in Equation (\ref{eq-H}),
where $V$ has compact support in $\mathring{\Sigma}\setminus \mathcal{N}$ and has small $W^{1,2}$-norm on $[0,1]\times \mathring{\Sigma}$; $X_{H_{V,f}}$ and $\phi^t_{H_{V,f}}$ are as before with respect to $H_{V,f}$.
Note that the wrapped Lagrangians $\phi^t_{H_V}(\sqcup_{i}T_{q_i}^*{\mathring{\Sigma}})\subset\mathring{\Sigma}_0\subset\mathring{\Sigma}$ for all $t\geq 0$, and hence cannot cross $\{s=0\}$. 

We follow the notation of Section \ref{subsection-wrap}. As in Definition \ref{def-CW}, we define:
\begin{definition}
    \label{def-CW-boundary}
    The wrapped Heegaard Floer cochain complex of $CW(\sqcup_{i}T_{q_i}^*{\mathring{\Sigma}})$ is $CF(\phi^1_{H_V}(\sqcup_{i}T_{q_i}^*{\mathring{\Sigma}}),\sqcup_{i}T_{q_i}^*{\mathring{\Sigma}})$.
\end{definition}

Choose sufficiently generic consistent collection of almost complex structures $\{J_{D_m}\}$ satisfying (J1)--(J3) as before and  
apply the rescaling argument of Section \ref{subsection-wrap} to $\{J_{D_m}\}$ to get $\{J''_{D_m}\}$. We also define $\tilde L''_j(D_m)$ as before.

Given $\mathbf{y}_1,\dots,\mathbf{y}_m\in CW(\sqcup_{i}T_{q_i}^*{\mathring{\Sigma}})$, let $\mathcal{M}(\mathbf{y}_1,\dots,\mathbf{y}_m,\mathbf{y}_0)$ be the moduli space of maps
\begin{equation*}
    u\colon(\dot F,j)\to(D_m\times T^*\mathring{\Sigma},J_{D_m}),
\end{equation*}
where $(F,j)$ is a compact Riemann surface with boundary and $u$ satisfies the conditions similar to (\ref{floer-condition}).

It is easy to check that all the conclusions in Section \ref{subsection-wrap} still hold except for Lemma \ref{lemma-compactness} and the $A_\infty$-relation. However, we claim that the standard proof of $A_\infty$-relation (see \cite[Proposition 4.0.3]{colin2020applications}) works by showing that:

\begin{lemma}
    \label{lemma-compactness-boundary}
Fix a sufficiently generic consistent choice of almost complex structures as above.
    Let $d=0\,\,\text{or}\,\,1$. Given generators $\mathbf{y}_1,\dots,\mathbf{y}_m\in CW(\sqcup_{i}T_{q_i}^*{\mathring{\Sigma}})$, $\mathcal{M}^{\mathrm{ind}=d,\chi}(\mathbf{y}_1,\dots,\mathbf{y}_m,\mathbf{y}_0)$ is empty for all but finitely many $\mathbf{y}_0$. When it is nonempty, $\mathcal{M}^{\mathrm{ind}=d,\chi}(\mathbf{y}_1,\dots,\mathbf{y}_m,\mathbf{y}_0)$ (and $\mathcal{M}^{\operatorname{ind}=d,\chi}(\mathbf{y}_1,\mathbf{y}_0)/\mathbb{R}$ if $m=1$) admits a compactification for each Euler characteristic $\chi$.
\end{lemma}
\begin{proof}
    We follow \cite[Lemma 2.41]{ganatra2020covariantly}. Let $\pi\colon T^*\mathring{\Sigma}\to\mathring{\Sigma}$ be the projection. For each $u\in\mathcal{M}(\mathbf{y}_1,\dots,\mathbf{y}_m,\mathbf{y}_0)$, we show that $\pi\circ u(\dot{F})\cap \{s<0\}=\varnothing$.
    
    Consider the holomorphic map 
    \begin{equation*}
        \pi_s\circ u\colon u^{-1}(\pi_s^{-1}(\mathbb{C}_{\mathrm{Re}<0}))\to\mathbb{C}_{\mathrm{Re}<0}.
    \end{equation*}
    By definition, $\pi\circ u(\partial\dot{F})\cap\{s\leq 0\}=\varnothing$, hence $u^{-1}(\pi_s^{-1}(\mathbb{C}_{\mathrm{Re}<0}))\subset F\setminus\partial F$ is an open subset of the interior of $F$. Therefore, $K\coloneqq(\pi_s\circ u)( u^{-1}(\pi_s^{-1}(\mathbb{C}_{\mathrm{Re}<0})))$ is an open subset of $\mathbb{C}_{\mathrm{Re}<0}$ by the open mapping theorem.
    
    On the other hand, $u^{-1}(\pi_s^{-1}(\mathbb{C}_{\mathrm{Re}\leq 0}))\subset F\setminus\partial F$ is a compact subset of the interior of $F$. Therefore, the function 
    \begin{equation*}
        \mathrm{Re}\circ\pi_s\circ u\colon u^{-1}(\pi_s^{-1}(\mathbb{C}_{\mathrm{Re}\leq 0}))\to\mathbb{R}
    \end{equation*}
    attains its minimum on $u^{-1}(\pi_s^{-1}(\mathbb{C}_{\mathrm{Re}\leq 0}))$. Since $K$ is open, this is only possible if $K=\varnothing$. Hence $u^{-1}(\pi_s^{-1}(\mathbb{C}_{\mathrm{Re}<0}))=\varnothing$ and then $\pi\circ u(\dot{F})\cap \{s<0\}=\varnothing$.
    
    We have shown how to prevent curves from crossing the vertical boundary of $T^*\mathring{\Sigma}_0$. The remaining proof is the same as that of Lemma \ref{lemma-compactness}.
\end{proof}

Hence $CW(\sqcup_{i}T_{q_i}^*{\mathring{\Sigma}})$ is a well-defined ordinary algebra supported in degree 0. 
~\\

To modify Section \ref{section-loop}, we consider paths in $\mathring{\Sigma}_0$ instead of ${\Sigma}$. We now use the metric $\tilde{f}^*g$. 
The Legendre transform $L_{f,v}$ is with $|v|_f$ instead of $|v|$ in (\ref{eq-lagrangian}), where $|\cdot|_f$ is the norm on $T\mathring{\Sigma}_0$ induced by $\tilde f^*g$.
Given $q_0,q_1\in \mathring{\Sigma}_0\setminus \mathcal{N}$, the Lagrangian action functional $\mathcal{A}_{V,f}$ is defined as in (\ref{eq-lagrangian-action}) with $L_V$ replaced by $L_{V,f}$.
By the choice of the metric $\tilde{f}^*g$ it is easy to see that
\begin{enumerate}
    \item no $V$-perturbed geodesics can exit $\mathring{\Sigma}_0$;
    \item the induced negative gradient flow of $\mathcal{A}_V$ always stays inside a compact region of $\mathring{\Sigma}_0$.
\end{enumerate}
Since $\tilde{f}\colon(\mathring{\Sigma}_0,\tilde{f}^*g)\to(\mathring{\Sigma},g)$ is an isometry, we have:
\begin{lemma}
   The Morse homology $HM_*(\Omega^{1,2}(\mathring{\Sigma}_0,q_0,q_1))$ induced by the metric $\tilde{f}^*g$ is well-defined and is isomorphic to $HM_*(\Omega^{1,2}(\mathring{\Sigma},q_0,q_1))$ induced by the metric $g$.  
\end{lemma}

It is easy to verify that all the definitions and conclusions in Section \ref{section-Hecke} hold for $\mathring{\Sigma}_0$. Hence the Hecke algebra $\mathrm{H}_{\kappa}(\mathring{\Sigma}_0,\mathbf{q})$ is well-defined.
~\\

Next we consider the modification of the map $\mathcal{F}$ defined by (\ref{eq-F}). We use the notation from Section \ref{section-f} with the following modifications:
\begin{enumerate}
    \item $H_V$ is replaced by $H_{V,f}$;
    \item choose a sufficiently generic consistent collection $T_{m-1}\mapsto J_{T_{m-1}}$ of compatible almost complex structures on $T_{m-1}\times T^*\mathring{\Sigma}$ for $T_{m-1}\in \mathcal{T}_{m-1}$ and all $m\geq2$, where $J_{T^*\mathring{\Sigma}}$ coincides with $j_m\times J_\mathcal{N}$ on $T_{m-1}\times T^*\mathcal{N}$.
\end{enumerate}
All of Section \ref{subsection-ev} carries over with the exception of Lemma \ref{lemma-H-T-1}. Recall that we denote the set of intersection points between $\sqcup_{i}T_{q_i}^*\mathring{\Sigma}$ (resp.\ $\phi^1_{H_V}(\sqcup_{i}T_{q_i}^*\mathring{\Sigma})$) and $\mathring{\Sigma}$ by $\mathbf{q}$ (resp.\ $\mathbf{q}'$). Let $\mathbf{y}\in CF(\phi^1_{H_V}(\sqcup_{i}T_{q_i}^*\mathring{\Sigma}),\sqcup_{i}T_{q_i}^*\mathring{\Sigma})$. We show that:

\begin{lemma}
    \label{lemma-compactness-boundary-H}
    The moduli space $\mathcal{H}^\chi(\mathbf{q}',\mathbf{y},\mathbf{q})$ admits a compactification for each Euler characteristic $\chi$.
\end{lemma}
\begin{proof}
    This is similar to the proof of Lemma \ref{lemma-compactness-boundary}. Let $\pi:T^*\mathring{\Sigma}\to\mathring{\Sigma}$ be the projection. For each $u\in\mathcal{H}^\chi(\mathbf{q}',\mathbf{y},\mathbf{q})$, consider the holomorphic map 
    \begin{equation*}
        \pi_s\circ u\colon u^{-1}(\pi_s^{-1}(\mathbb{C}_{\mathrm{Re}<0}))\to\mathbb{C}_{\mathrm{Re}<0}.
    \end{equation*}
    By definition, $u(\partial\dot{F})\cap T^*\mathcal{N}|_{s<0}$ is a subset of the zero section $\mathcal{N}|_{s<0}$. Since $u^{-1}(\pi_s^{-1}(\mathbb{C}_{\mathrm{Re}\leq 0}))\subset F\setminus\partial F$ is an compact closed subset of $F$, its image $P\coloneqq(\pi_s\circ u)(u^{-1}(\pi_s^{-1}(\mathbb{C}_{\mathrm{Re}\leq 0})))\subset\mathbb{C}$ is also compact. 
    
    Note that $(\pi_s\circ u)(u^{-1}(\pi_s^{-1}(\mathbb{C}_{\mathrm{Re}<0,\,\mathrm{Im}\neq0})))\subset\mathbb{C}$ is open by the open mapping theorem. As a result, $\partial P\cap\mathbb{C}_{\mathrm{Re}<0,\,\mathrm{Im}\neq0}=\varnothing$, hence $P\subset\mathbb{C}_{\mathrm{Re}=0}\cup \mathbb{C}_{\mathrm{Im}=0}$. This implies $u^{-1}(\pi_s^{-1}(\mathbb{C}_{\mathrm{Re}\leq 0}))\subset \partial F$, which is only possible if $P\subset{\mathbb{C}_{\mathrm{Re}=0}}$. Therefore we conclude that $\pi\circ u(\dot{F})\cap \{s<0\}=\varnothing$.
    
    We have shown that curves in $\mathcal{H}^\chi(\mathbf{q}',\mathbf{y},\mathbf{q})$ cannot cross the vertical boundary of $T^*\mathring{\Sigma}_0$. The remaining proof is the same as that of Lemma \ref{lemma-H-T-1}.
\end{proof}

The parallel modification of Lemma \ref{lemma-H-T-2} and its proof are similar. Also note that Proposition \ref{prop-algebra} still holds. Therefore
\begin{equation*}
    \mathcal{F}\colon CW(\sqcup_{i}T_{q_i}^*\mathring{\Sigma}) \to  \mathrm{H}_\kappa(\mathring{\Sigma}_0,\mathbf{q})|_{c=1}\otimes_{\mathbb{Z}[\hbar]} \mathbb{Z}[[\hbar]]\cong\mathrm{H}_\kappa(\mathring{\Sigma},\mathbf{q})|_{c=1}\otimes_{\mathbb{Z}[\hbar]} \mathbb{Z}[[\hbar]]
\end{equation*}
is well-defined.

To modify the proof of Theorem \ref{thm-main}, it suffices to modify the proof of Lemma \ref{lemma-iso-h-0} when $\kappa=1$. Note that $\hbar=0$ is automatically satisfied in this case. 
\begin{lemma}
    \label{lemma-iso-h-0-boundary}
    When $\kappa=1$, the map $\mathcal{F}$ is an isomorphism:
    \begin{equation*}
        \mathcal{F}\colon HW(T_{q}^*\mathring{\Sigma})\to \mathrm{H}_1(\mathring{\Sigma}_0,q)|_{c=1}\cong \mathrm{H}_1(\mathring{\Sigma},q)|_{c=1}.
    \end{equation*}
\end{lemma}
\begin{proof}
    This is essentially Lemma 5.1 of \cite{abouzaid2012wrapped}, replacing $\Sigma$ by $\mathring{\Sigma}$. 
    
    Recall that Abbondandolo and Schwarz \cite[Theorem 3.1]{abbondandolo2006floer} constructed a
    chain isomorphism
    \begin{equation*}
        \Theta\colon CM_*(\Omega^{1,2}({\Sigma},q))\to CW(T_{q}^*{\Sigma})
    \end{equation*}
    by a specific moduli space $\mathcal{M}_\Omega^+$ of holomorphic curves of index 0 (see \cite[p.35]{abbondandolo2006floer}). Abouzaid then showed that $\mathcal{F}$ is a homotopy inverse of $\Theta$ by constructing another moduli space $\mathcal{C}$ of curves of index 1 (see \cite[p.33]{abouzaid2012wrapped}). 
    
    In the case of $\mathring{\Sigma}$, we define the moduli spaces $\mathcal{M}_\Omega^+$ and $\mathcal{C}$ in a similar manner. Again it suffices to show that no curves in $\mathcal{M}_\Omega^+$ or $\mathcal{C}$ can cross the vertical boundary of $T^*\mathring{\Sigma}_0$. We omit the details which are similar to the proofs of Lemma \ref{lemma-compactness-boundary} and Lemma \ref{lemma-compactness-boundary-H}.
\end{proof}



\printbibliography

\end{document}